\documentclass[11pt]{article}
\usepackage{amsmath,amssymb,amsthm
	,makeidx,fancyhdr}

\usepackage[dvips]{color}
\usepackage{amsmath,amssymb,amsthm
	,makeidx,fancyhdr
}
\usepackage{color}
\usepackage[all]{xy}
\usepackage[latin1]{inputenc}
\usepackage{tikz}

\theoremstyle{definition}
\newtheorem{definition}{Definition}[section]

\theoremstyle{plain}
\newtheorem{theorem}[definition]{Theorem}

\newtheorem{proposition}[definition]{Proposition}
\newtheorem{lemma}[definition]{Lemma}
\newtheorem{corollary}[definition]{Corollary}

\newtheorem{remark}[definition]{Remark}

\newcommand{\qo}{\mathbb{Q}}

\newcommand{\fr}{{}^\frown}

\newcommand{\name}{\dot}
\newcommand{\can}{\check}
\newcommand{\force}{\Vdash}

\newcommand{\la}{\langle}
\newcommand{\ra}{\rangle}
\newcommand{\elem}{\prec}
\newcommand{\uhr}{\restriction}
\newcommand{\dhr}{\downharpoonright}

\newcommand{\power}{\mathcal{P}}
\newcommand{\rar}{\rightarrow}

\newcommand{\po}{\mathbb{P}}

\newcommand{\A}{\mathfrak{A}}

\renewcommand{\iff}{\leftrightarrow}

\newcommand{\lra}{\leftrightarrow}
\newcommand{\Lan}{\mathcal{L}}
\newcommand{\AP}{\mathcal{A}}
\DeclareMathOperator{\SCH}{SCH}

\DeclareMathOperator{\cf}{cf}

\DeclareMathOperator{\SK}{SK}
\DeclareMathOperator{\rng}{rng}

\DeclareMathOperator{\cof}{Cof}

\DeclareMathOperator{\dom}{dom}
\DeclareMathOperator{\GCH}{GCH}

\DeclareMathOperator{\otp}{otp}

\DeclareMathOperator{\tp}{tp}
\DeclareMathOperator{\col}{Coll}

\newcommand{\bP}{\bar{\po}}
\newcommand{\bp}{\bar{p}}
\newcommand{\bq}{\bar{q}}
\newcommand{\bt}{\bar{t}}
\newcommand{\bro}{\bar{\rho}}
\newcommand{\vro}{\vec{\rho}}

\newcommand{\np}[1]{\bar{#1}}   

\newcounter{saveenumi}

\title{On Singular Stationarity II \\(tight stationarity and extenders-based methods)}
\author{Omer Ben-Neria}
\date{}

\begin{document}
	\maketitle
	
	\begin{abstract}
    We study the notion of tightly stationary sets which was introduced by Foreman and Magidor in \cite{ForMag-MS}.
We obtain two consistency results which show that it is possible for a sequence of regular cardinals $\la \kappa_n \ra_{n < \omega}$ to have the property that for every sequence $\vec{S}$, of some fixed-cofinality stationary sets $S_n \subseteq \kappa_n$, $\vec{S}$ is tightly stationary in a generic extension. 
The results are obtained using variations of the short-extenders forcing method.
	\end{abstract}
	
	\section{Introduction}
	This paper is a contribution to the study of singular stationarity.
	We prove two consistency results concerning the notion of tightly stationary sets, which was introduced by Foreman and Magidor in the 1990s.
	Two notions of singular stationarity were introduced and studied in \cite{ForMag-MS}: Mutual Stationarity and Tight Stationarity. Both notions are related to properties of sequences $\la S_i \mid i <\tau\ra$ of sets $S_i \subseteq \kappa_i$ for an increasing sequence of regular cardinals $\la \kappa_i \mid i < \tau\ra$. 
	It is shown that the notion of tight stationarity sets is a strengthening of mutually stationary which satisfies analogs of the well-known Fodor's Lemma and Solovay's splitting theorem for stationary sets of regular cardinals. Foreman and Magidor raised the question of whether every mutually stationary sequence is tightly stationary. 
	Models containing mutually stationary sequences on the cardinals $\la \omega_n \mid n < \omega\ra$ which are not tight were obtained by Cummings, Foreman, and Magidor, and by Steprans and Foreman (\cite{CFM-CanStrII}).
	Chen and Neeman (\cite{CheNee-TS}) obtained a strong global result of a model in which there are mutual and non-tight stationary sequences on every increasing $\omega$-sequence of regular cardinals. Moreover, they show that the property of their model is immune to further forcing by a wide class of natural posets. 
	
	In \cite{BN-MSI}, several consistency results regarding mutually stationary were obtained by the author.
	The goal of this paper is to introduce two positive results concerning tightly stationary sets.
	We construct models which contain a sequence of regular cardinals $\vec{\kappa} = \la \kappa_n \mid n < \omega\ra$ with the property that every sequence $\vec{S} = \la S_n \mid n < \omega\ra$ of stationary sets of some fixed cofinality $S_n \subseteq \kappa_n$ is tightly stationary in a forcing extension. This shows there is no natural indestructible obstruction to constructing models in which the notions of mutual stationarity and tight stationarity coincide. 
	
	The methods we apply to obtain the results are called extender-based forcing methods. They are known for their ability to  generically add scales to products of cardinals 
	$\prod_n \kappa_n$, which are known to be connected with tightly stationary sequences. 
	In \cite{CFM-CanI}, the authors have established many connections between scales and tight structures. 
	These connections have been further studied and extender in \cite{chen-treelikeTS}, and obtained strong failure results of tight stationarity.
	
	Building on the known connections between tight stationarity and scales, we are able to reduce the problem of forcing a sequence $\la S_n \mid n < \omega\ra$ to be tightly stationary to  forcing a scale $\vec{f} =  \la f_\alpha \mid \alpha < \lambda\ra$ in $\prod_n\kappa_n$ with certain properties (i.e., scales with stationarily many good continuous points $\delta < \lambda$ for which $f_\delta(n) \in S_n$ for all but finitely many $n < \omega$).
	This approach leads us to the following two results.
	\begin{theorem}\label{thm1}
		Suppose that $\la \kappa_n \mid n < \omega\ra$ is an increasing sequence of $(+1)$-extendible cardinals. Then for every sequence of fixed-cofinality stationary sets $\vec{S} = \la S_n \mid n < \omega\ra$ with $S_n \subseteq \kappa_n$, there exists a generic extension in which $\vec{S}$ is tightly stationary.
	\end{theorem}
	
	\begin{theorem}\label{thm2}
		It is consistent relative to the existence of a sequence $\la \kappa_n \mid n < \omega\ra$ of cardinals $\kappa_n$ which are $\kappa_n^{+n+3}$-strong, that there is a model with a subset $\la \omega_{s_n} \mid n < \omega\ra$ of the $\omega_n$'s such that every fixed-cofinality sequence of stationary sets $S_n \subseteq \omega_{s_n}$ is tightly stationary in a generic extension.
	\end{theorem}
	
	Although it might seem, at first glance, that the second theorem is superior to the first one in every parameter, the second theorem provides a result which is less canonical in the following sense:
	As opposed to the first theorem, where the sequence of cardinal $\la \kappa_n \mid n < \omega\ra$ is given in advance, the sequence of cardinals $\la \omega_{s_n}\mid n < \omega\ra$ in the second theorem does not exist in the minimal ground model (i.e., in the core model $K$ or the mantle of the final generic extension) but rather obtained as a Prikry generic sequence in some intermediate extension.
	
	Both theorems are obtained using variants of the short-extenders forcing method of Gitik (\cite{Gitik-EBF1} and \cite{GitUng-SEF}). 
	To prove Theorem \ref{thm1}, we apply the short-extenders method to a sequence of long extenders\footnote{i.e., a $(\kappa_n,j_n(\kappa_n))$-extender derived from an embedding $j_n$ with critical point $\kappa_n$.} and argue that if the extenders are derived from $(+1)$-extendible embeddings then every stationary sequence in the ground model is tight in the generic extension.
	For Theorem \ref{thm2}, we force with variants of the short-extenders forcing where the $n$th assignment function $a_n$ depends on the generic information of the previous Prikry points of the normal generators. \\

	\textbf{Organization of the paper - }  In Section \ref{section-pre}, we review relevant results which connect the notions of approachable sets, tight structures, and scales. The description follows the work of Cummings, Foreman, and Magidor from \cite{CFM-CanI}. We conclude Section \ref{section-pre} with a result (Proposition \ref{prop-IAmain}) which reduces the problem of obtaining tightly stationary sequences to the existence of a certain scale.
	In Section \ref{section-thm1}, we describe Gitik's extenders-based forcing and its main properties. We then show that a long-extender variant of this forcing produces a desirable scale and prove Theorem \ref{thm1}.
	Finally, in Section \ref{section-thm2}, we introduce a variant of the short-extenders forcing which allows us to obtain a similar construction below $\aleph_\omega$ and prove Theorem \ref{thm2}.${}$\\
	
	Our notations are (hopefully) standard with the exception that our forcing convention follows the Jerusalem forcing convention. This means that for two conditions $p,q$ in poset $(\po,\leq)$, the fact $p$ is stronger than $q$ (i.e., it is more informative) will be denoted by $p \geq q$. 
	
	\section{Preliminaries}\label{section-pre}
	
	The purpose of this section is to define the notions of tight structures and tightly stationary sequences and review their connections with the existence of scales with many good points. 
	The connections go through the notions of internally approachable structures and results by Cummings, Foreman, Magidor, and Shelah.
	Our presentation follows the description of \cite{CFM-CanI}, and we refer the reader to this paper for an extensive treatment of the subject. 
	
	We commence by defining the notions of tight structures and tightly stationary sequences from \cite{ForMag-MS}.
	
	\begin{definition}
		\begin{enumerate}
			\item 
			Let $\vec{\kappa} = \la \kappa_n \mid n < \omega\ra$ be an increasing sequence of regular cardinals and $\A$ an algebra expanding $\la H_\theta,\in,<_\theta\ra$ for some regular cardinal $\theta > \cup_n \kappa_n$. A subalgebra $M \elem \A$ is called tight for $\vec{\kappa}$ if $\vec{\kappa} \in M$ and for every $g \in \prod_n (M \cap \kappa_n)$ there exists $f \in M \cap \prod_n \kappa_n$ such that $g(n) < f(n)$ for all $n < \omega$.
			
			\item
			A stationary sequence in $\vec{\kappa}$ is a sequence $\vec{S} = \la S_n \mid n < \omega\ra$ so that $S_n \subseteq \kappa_n$ is stationary for every $n < \omega$.
			$\vec{S}$ is tightly stationary if for every algebra $\A$ there exists a tight substructure $M \elem \A$ such that $\sup(M \cap \kappa_n) \in S_n$ for every $n< \omega$ (we say that $M$ meets $S_n$).
		\end{enumerate}
	\end{definition}
	It is not difficult to see that in the definition of a tight structure, we can replace the last requirement with a slightly weaker one, which demands that for
	every $g \in \prod_n (M \cap \kappa_n)$ there exists $f \in M \cap \prod_n \kappa_n$ such that $g(n) < f(n)$ for all but finitely many $n < \omega$ (also denoted by $g<^*f$).
	
	We will focus on the case where the sets $S_n$ in $\vec{S}$ consist of ordinals of some fixed cofinality. 
	That is, sequences $\vec{S}$ for which there exists some regular $\mu < \cup_n \kappa_n$, 
	such $S_n$ is defined for every $n$ with $\mu < \kappa_n$ and $S_n \subseteq \kappa_n \cap \cof(\mu)$. 
	To show that a certain sequence $\vec{S}$ is tightly stationary we will show that every algebra $\A$ has a tight subalgebra $M \elem \A$ of size $|M| = \mu$, such that 
	$\sup(M \cap \kappa_n) \in S_n$ for all but finitely many $n < \omega$. 
	Obtaining this suffices to show that $\vec{S}$ is tightly stationary since by 
	a well-known argument of Baumgartner (\cite{Baum}), adding ordinals to $M$ below a cardinal $\kappa_i$ does not change its supremum below any regular cardinal $\kappa > \kappa_i$ in $M$. 
	The same argument shows that this addition does add a new function $f \in \prod_n \kappa_n$ which dominates every function in $M \cap \prod_n \kappa_n$ in the $\kappa_{\omega} = \cup_n\kappa_n$ directed order $<^*$, which is defined by $f <^* g$ if and only if $f(n) < g(n)$ for all but finitely many $n < \omega$.
	It follows that for every finite sequence of stationary sets $S_m, S_{m+1},\dots,S_k$ with $S_i \subseteq \kappa_i \cap \cof(\mu)$, there exists an elementary extension $M'$ of 
	$M$, which meets $S_n$ for all $n$ with $\kappa_n > \mu$, and further satisfies
	\begin{itemize}
		\item $\sup(M' \cap \kappa_n) =\sup(M \cap \kappa_n)$ for almost all $n < \omega$.
		\item Every function in $M' \cap \prod_n \kappa_n$ is dominated in $<^*$ by a function in $M \cap \prod_n \kappa_n$.
	\end{itemize}
	The last, combined with the fact $M$ is tight, guarantees $M'$ is also tight. 
	
	\begin{remark}
		The same considerations show that the ideal-based methods which were used in \cite{BN-MSI} to obtain mutually stationary sequences are not useful in the context of tight stationarity. The reason is that substructures $M \elem \A$ constructed in $\cite{BN-MSI}$ are limits of $\omega$-chains of structures, $M = \cup_n M_n$, 
		where $M_0$ is a tight structure, and for each $n < \omega$, $M_{n+1}$ is obtained from $M_n$  by adding a family $s(\alpha)$ of sets of ordinals below $\kappa_{n+1}$ (i.e, $M_{n+1}  = \SK^{\A}(M_n \cup s(\alpha))$) so that $M_{n+1} \cap \kappa_n = M_n \cap \kappa_n$ and $\sup(M_{n+1} \cap \kappa_{n+1}) \in S_{n+1}$. 
		While it is clear that $\sup(M_n \cap \kappa_n) > \sup(M_0 \cap \kappa_n)$ for almost all $n$,  it is possible to show by induction on $n$, using Baumgartner's argument, that $M_0 \cap \prod_n \kappa_n$ is $<^*$-cofinal in the product $M_k \cap \prod_n \kappa_n$ for all $k$, and thus also in $M \cap \prod_n \kappa_n$. 
		Consequently, the functions in $M \cap \prod_n \kappa_n$, which are dominated by the functions in $M_0 \cap \prod_n \kappa_n$, cannot dominate all function in the product $\prod_n (M \cap \kappa_n)$, which is strictly bigger than the product 
		$\prod_n (M_0 \cap \kappa_n)$.
	\end{remark}
	
	We proceed to describe the connection between tight structures, approachable ordinals, and scales.

	\subsection{Internally approachable structures}
	
	\begin{definition}
		Let $\A$ be an algebra expanding $\la H_\theta,\in,<_\theta\ra$ for some regular cardinal $\theta$.
		\begin{enumerate}
			\item A sequence $\vec{M} = \la M_i \mid i < \rho\ra$ of substructures $M_i \elem \A$ is called an internally approachable chain if it is $\subseteq$-increasing and continuous\footnote{i.e., $M_i \subseteq M_j$ for all $i < j$, and $M_\gamma = \bigcup_{i<\gamma}M_i$ if $\gamma < \rho$ is a limit ordinal.}, and for every successor ordinal $j < \delta$, $\vec{M}\uhr j = \la M_i \mid i < j\ra$ belongs to $M_j$.
			\item A substructure $M \elem \A$ is called internally approachable (IA) if there exists an IA chain $\la M_i \mid i < \rho\ra$ such that $M = \bigcup_{i < \rho}M_i$. 
		\end{enumerate}
	\end{definition}
	
	We refer to $\rho$ in the definition as the length of the IA chain. 
	
	
	Internally approachable structures satisfy many natural properties. We list three.
	\begin{theorem}[\cite{CFM-CanI}]
		Suppose that $\delta$ is a regular uncountable cardinal and that $M$ is a limit of an IA chain $\vec{M} = \la M_i \mid i < \delta\ra$. Then
		\begin{enumerate}
			\item $\delta \subseteq M$.
			\item For every regular cardinal $\kappa  < \delta$ in $M$, $\cf(M \cap \kappa) = \delta$. 
			\item 
			Suppose that $\vec{\kappa} = \la \kappa_n \mid n < \omega\ra$ is an increasing sequence of regular cardinals. If $\vec{\kappa} \in M$ and $|M| < \kappa_0$ then every function in $\prod_n (M \cap \kappa_n)$ is pointwise dominated by a function in $M \cap \prod_n \kappa_n$. Therefore, $M$ is tight for $\vec{\kappa}$.
		\end{enumerate}
	\end{theorem}
	Regarding the third statement, we note that since $M$ is a limit of an IA chain $\vec{M} = \la M_i \mid i < \delta\ra$ of uncountable length $\delta$, then the range of every function $f \in \prod_n (M \cap \kappa_n)$ is contained in $M_i$ for some $i <\delta$. 
	Thus $f$ is dominated by the characteristic function $\chi_{M_i}^{\vec{\kappa}}$ of $M_i$, defined by 
	$\chi_{M_i}^{\vec{\kappa}}(n) = \sup(M_i \cap \kappa_n)$. This function clearly belongs to $M \cap \prod_n \kappa_n$ since $M_i$ does.

	\subsection{Existence of IA structures}
	We state two results by Shelah and by Foreman and Magidor, which together, guarantee the existence of many ordinals $\delta$ which are of the form $\sup(N \cap \lambda)$ form some IA structure $N$ of size $|N| = \cf(\delta)$.
	
	\begin{definition}\label{def-IApoints}
		Let $\lambda$ be a regular cardinal and $\vec{a} = \la a_\nu \mid \nu < \lambda\ra$ be a sequence of bounded subsets of $\lambda$. A limit ordinal $\delta < \lambda$ is said approachable with respect to $\vec{a}$ if there is a cofinal subset $D \subseteq \delta$ of minimal ordertype $\otp(D) = \cf(\delta)$, such that for every $\beta < \delta$, $D \cap \beta \in \vec{a}\uhr\delta = \la a_\nu \mid \nu < \delta\ra$. We denote the set of approachable ordinals with respect to $\vec{a}$ by $S_{\vec{a}}$.
	\end{definition}
	

	\begin{theorem}[Shelah,\cite{She-CarAri}]\label{thm-ShelahIA}
		Suppose that $\lambda = \eta^+$ is a successor cardinal. Then for every regular cardinal $\mu < \eta$ there exists a sequence $\vec{a} \subseteq [\lambda]^{<\lambda}$ 
		such that $S_{\vec{a}} \cap \cof(\mu)$ is stationary in $\lambda$.
	\end{theorem}

	Foreman and Magidor established the connection between approachable ordinals and internally approachable structures. 
	\begin{theorem}[Foreman-Magidor \cite{ForMag-veryweak}]\label{thm-IAmodels}
		Let $\lambda$ be a regular cardinal 
		and $\vec{a}  = \la a_\alpha \mid \alpha < \lambda\ra$ be a sequence of bounded subsets of $\lambda$. 
		Suppose that $\A$ is an algebra which expands $\la H_\theta,\in,<_\theta, \vec{a} \ra$ for some regular $\theta > \lambda$. Then there exists a closed unbounded set $C \subseteq \lambda$ such that for every $\delta \in C \cap S_{\vec{a}}$ there is an IA chain of length $\mu = \cf(\delta)$ whose limit $M \elem \A$ has cardinality $\mu$ and satisfies that $\sup(M \cap \lambda) = \delta$.
	\end{theorem}
	The idea is to start with a long IA chain  $\la M_i \mid i < \lambda\ra$ of substructures
	of $\A$ and take $C$ to be the club of $\delta < \lambda$ such that $M_\delta \cap \lambda = \delta$. 
	Then, for $\delta \in C \cap S_{\vec{a}}$, $M_\delta$ contains $\vec{a}\uhr \delta$ and we can therefore approximate some cofinal $D \subseteq \delta$ of ordertype $\mu= \cf(\delta)$ in $M_\delta$. With this, one can create an IA-chain of $\mu$-sized structures $\la N_\nu \mid \nu < \mu\ra$ within $M_\delta$, each of which is the Skolem hull in some $M_i$ of initial segments of $D$. The union of these substructures is an IA-structure $N \subseteq M_\delta$ of size $\mu$, which contains $D$ and thus satisfies $\sup(N \cap \lambda) = \delta$.

	\begin{corollary}\label{cor-IAsum}
		Suppose that $\lambda = \eta^+$ is a successor cardinal and $\mu < \eta$ is regular. Then for every regular cardinal $\theta  >\lambda$ and an algebra $\A$ expanding 
		$\la H_\theta,\in,<_\theta\ra$ there is a sequence $\vec{a}$ such that $S_{\vec{a}} \cap \cof(\mu)$ is stationary in $\lambda$ and for every $\delta \in S_{\vec{a}}$ there is an IA substructure $M \elem \A$ of length and cardinality $\mu$ such that $\sup(M \cap \lambda) = \delta$.
	\end{corollary}
	

	\subsection{Scales and IA structures}
	
	Let $<^*$ be the order relation on functions $f$ from $\omega$ to the ordinals, defined by $f<^* g$ if and only if $f(n) < g(n)$ for all but finitely many $n < \omega$. 
	\begin{definition}
		Let $\vec{\kappa} = \la \kappa_n \mid n < \omega\ra$ be an increasing sequence of regular cardinals.
		A sequence of functions $\vec{f} = \la f_\alpha \mid \alpha < \lambda\ra$ is called a scale on $\prod_n \kappa_n$ if  it satisfies the following conditions:
		\begin{enumerate}
			\item $\vec{f}$ is increasing in the order $<^*$;
			\item $\vec{f}$ is cofinal in the structure $(\prod_n \kappa_n,<^*)$ in the sense that for every $g \in \prod_n \kappa_n$ there exists some $\alpha < \lambda$ such that $g<^* f_\alpha$; and
			\item for every $\alpha < \lambda$, $f_\alpha(n) < \kappa_n$ for all but finitely many $n < \omega$. 
		\end{enumerate}
	\end{definition}
	Our definition of a scale is a slight relaxation of the usual definition of a scale, which further requires that $\vec{f}$ to be contained in $\prod_n \kappa_n$ (namely, that $f_\alpha(n) < \kappa_n$ for all $n < \omega$). The two versions are equivalent for all of our purposes and it is not difficult to transform a sequence $\vec{f}$ which is a scale according to our definition to a scale accordring to the standard definition.
	
	We proceed to define exact upper bounds and continuity points of scales.
	\begin{definition}
		Let $\vec{f} = \la f_\alpha \mid \alpha< \lambda\ra$ be a scale on some product $\prod_n \kappa_n$.
		\begin{enumerate}
			\item Let $\delta < \lambda$ and $\vec{f}\uhr \delta = \la f_\alpha \mid \alpha < \delta\ra$. We say that a function $g \in \prod_n \kappa_n$ is an exact upper bound (eub)  of $\vec{f}\uhr \delta$ if $\vec{f}\uhr \delta$ is a scale on $\prod_n g(n)$.
			\item We say that an ordinal $\delta < \lambda$ is a continuity point of $\vec{f}$ if either $\vec{f}\uhr\delta$ does not have an eub, or $f_\delta$ is such a bound.
		\end{enumerate}
	\end{definition}
	
	It is not difficult to verify that if $g_1$ and $g_2$ are two eubs of $\vec{f}\uhr\delta$ then $g_1(n) = g_2(n)$ for almost all $n < \omega$.
	
	Let $\A$ be an algebra expanding $\la H_\theta,\in,<_\theta\ra$ for some regular cardinal $\theta$, and $M \elem \A$ be a tight substructure which contains a scale $\vec{f} = \la f_\alpha \mid \alpha < \lambda\ra$ on $\prod_n \kappa_n$. Denote $\sup(M \cap \lambda)$ by $\delta$.
	We make a few observations concerning $\vec{f}\uhr \delta$ and $M$.
	\begin{enumerate}
		\item Every function in $M \cap \vec{f}$ is $<^*$-dominated by the characteristic function $\chi_M^{\vec{\kappa}}$ of $M$, defined by $\chi_{M}^{\vec{\kappa}}(n) = \sup(M \cap \kappa_n)$.
		\item Suppose $h$ is a function in $\prod_n \chi_M^{\vec{\kappa}}(n)$. Then $h$ is pointwise dominated by a function $g \in \prod_n (M \cap \kappa_n)$. Now, since $M$ is tight, $g$ is $<^*$-dominated by some $f \in M \cap \prod_n \kappa_n$, which, in turn, is $<^*$-dominated by some $f_\alpha \in M \cap \vec{f}$. 
		\item It follows from the last two observations that $\chi_M^{\vec{\kappa}}$ is an eub of $\vec{f} \cap M$. Since $\vec{f} \cap M$ is cofinally interleaved in the ordering $<^*$ with $\vec{f}\uhr \delta$, we conclude that $\chi_{M}^{\vec{\kappa}}$ is an eub of $\vec{f}\uhr \delta$.
	\end{enumerate}
	
	By combining the last observation with Corollary \ref{cor-IAsum}, and the fact that every IA structure is tight, we obtain the following conclusion.
	\begin{proposition}\label{prop-IAmain}
		Let $\vec{\kappa} = \la \kappa_n \mid n < \omega\ra$ be an increasing sequence of regular cardinals whose limit is $\kappa_{\omega} = \cup_n \kappa_n$. Suppose that $\vec{f} = \la f_\alpha \mid \alpha < \lambda\ra$ is a scale on $\prod_n \kappa_n$ of a successor length $\lambda \geq \kappa_{\omega}^+$ and that $\A$ is an algebra expanding $\la H_\theta, \in,<_\theta, \vec{f}\ra$ for some regular cardinal $\theta > \lambda$. 
		Then for every regular cardinal $\mu < \kappa_\omega$ there is a sequence $\vec{a} \subseteq [\lambda]^{<\lambda}$ and a closed unbounded set $C \subseteq \lambda$ such that
		$S_{\vec{a}} \cap \cof(\mu)$ is stationary in $\lambda$ and for every $\delta \in S_{\vec{a}} \cap C$ there is a tight substructure $M \elem \A$ which satisfies that $\sup(M \cap \lambda) = \delta$ and $\chi_M^{\vec{\kappa}}$ is an eub of $\vec{f}\uhr \delta$. If moreover, $\delta$ is a continuity point of $\vec{f}$ then 
		$\sup(M \cap \kappa_n) = f_\delta(n)$ for almost every $n < \omega$. 
	\end{proposition}
	

	\section{Short extenders forcing and tight stationarity}\label{section-thm1}
	The purpose of this section is to prove Theorem \ref{thm1}. The proof is obtained by forcing with a version of Gitik's short extenders-based forcing with certain long extenders. 
	Our presentation follows \cite{GitUng-SEF} for the most part and omits most of the technical proofs. The only exception to this
	is that we will replace the notion of $k$-good ordinals in \cite{GitUng-SEF} with the more recent one from \cite{Gitik-EBF1}.
	We commence by describing the large cardinal framework which is used to construct the forcing.
	
	\subsection{Ground model assumptions and related forcing preliminaries}
	Let $\kappa$ be a measurable cardinal and $j: V_{\kappa+1} \to N$ be an elementary embedding of transitive sets with critical point $\kappa$ such that ${}^\kappa N \subseteq N$. 
	We say 
	\begin{enumerate}
		\item  $j$ is $\lambda$-strong for some $\lambda \leq j(\kappa)$ if $V_{\lambda} \subseteq N$.
		\item  $j$ is $(+1)$-extendible if $N = V_{j(\kappa)+1}$. 
	\end{enumerate}
	
	Correspondingly, we say
	\begin{itemize}
		\item $\kappa$ is $\lambda$-strong if there exists a $\lambda$-strong embedding $j$ as above, with $\lambda < j(\kappa)$.
		\item $\kappa$ is  superstrong if there exists an embedding $j$ as above which is $j(\kappa)$-strong. 
		\item $\kappa$ is $(+1)$-extendible if there exists a $(+1)$-extendible embeding $j$ as above.
	\end{itemize}
	
	These notions are consistency-wise increasing in the large cardinal hierarchy: A $(+1)$-extendible cardinal is superstrong, and the consistency of a superstrong cardinal implies the consistency of a cardinal $\kappa$ which is $\lambda$-strong for every $\lambda$. 
	Furthermore, if $\kappa$ is $\kappa^+$ supercompact cardinal, then it is a limit of $(+1)$-extendible cardinals (see \cite{kanamori}).

	For the rest of the section, we assume $V$ is a model of $\GCH$ which contains an increasing sequence of cardinals $\vec{\kappa} = \la \kappa_n \mid n < \omega\ra$ such that each $\kappa_n$ is the critical point of an elementary embedding $j_n : V_{\kappa_n+1} \to N_n$ which is $\lambda_n$-strong for some regular $\lambda_n$ with $\kappa_n^{+n+2} \leq \lambda_n \leq j_n(\kappa_n)$. 
	We also fix a regular cardinal $\chi >> \kappa_{\omega} = \cup_n \kappa_n$ and a structure $(H_{\chi}, \in,<_\chi)$. 
	For each $n < \omega$, we derive an $(\kappa_n,\lambda_n)$-extender $E_n$ from $j_n$ as follows. For every ordinal $\alpha \in [\kappa_n,\lambda_n)$ let $E_n(\alpha)$ be the $\kappa_n$-complete measure on $\kappa_n$ defined by $X \in E_n(\alpha) \iff  \alpha \in j_n(X)$. We define a Rudin-Kiesler order on the indicies of $E_n$ by writing $\alpha \leq_{E_n} \beta$ if and only if $\alpha \leq \beta$ and there exists a function $f : \kappa_n \to \kappa_n$ so that $j_n(f)(\beta) = \alpha$. 
	For each $\alpha \leq_{E_n} \beta$ we denote the first function $f$ with the above property in the well ordering $<_\chi$  by $\pi_{\beta,\alpha}$, with the possible exception when $\alpha = \beta$, in which case we take $\pi_{\beta,\beta}$ to be the identity function. 
	It turns out that the ordering $\leq_{E_n}$ is $\kappa_n$-directed (\cite{Gitik-HB}).
	
	We proceed to define the notion of $k$-good indices. Our notion here deviates from the description of \cite{GitUng-SEF}, and follows \cite{Gitik-EBF1}.
	\begin{definition}\label{def-goodpoints}
		For any two integer values $1 < k \leq n$, let $\A_{n,k}$ be the structure  $( H_{\chi^{+k}}, \in, <_{\chi^{+k}}, \chi,E_n, \la \alpha \mid \alpha \leq \kappa_n^{+k}\ra)$\footnote{namely, $\A_{n,k}$ is the expansion of $( H_{\chi^{+k}}, \in, <_{\chi^{+k}}, \chi,E_n)$, in a language which contains  $\kappa_n^{+k}$ additional constant symbols, $c_\alpha$, $\alpha < \kappa_n^{+k}$, so that each $c_\alpha$ is interpreted in the model $\A_{n,k}$ as $\alpha$.}.
		We assume that the well-ordering $<_{\chi^{+k}}$ of $H_{\chi^{+k}}$ extends the given order $<_\chi$ of $H_{\chi}$.
		An ordinal $\delta < \lambda_n$ is called $k$-good if there exists an elementary substructure $M_{n,k}(\delta) \elem \A_{n,k}$ so that $M_{n,k}(\delta) \cap \lambda_n = \delta$. $\delta$ is said to be good if it is $k$-good for every $k \leq n$. 
	\end{definition}
	
	It is easy to see that the set of good ordinals $\alpha$ is closed unbounded in $\lambda_n$ for each $n < \omega$. 
	We end this part with a simple but important observation. 
	\begin{lemma}\label{lem-Observationsgoodness}
		$\A_{n,l} \in \A_{n,k}$ for every $n < \omega$ and $l < k < \omega$. We therefore have
		\begin{enumerate}
			\item An ordinal $\gamma < \lambda_n$ is $l$-good iff $\A_{n,k} \models \gamma \text{ is } l\text{-good }$. 
			\item Suppose that $\gamma$ is $k$-good and $x \in \A_{n,k}$ is a set of ordinals with $\min(x) \geq \gamma$. For every formula $\phi(v)$ in parameters from
			$M_{n,k}(\gamma)$, if $\A_{n,k} \models \phi(x)$ then for every $\gamma' < \gamma$ there exists a set of ordinals $x' \in M_{n,k}(\gamma)$ (in particular $x' \subseteq \gamma)$ such that $\min(x') > \gamma'$ and $M_{n,k}(\gamma) \models \phi(x')$. 
		\end{enumerate}
	\end{lemma}
	
	Assuming $\GCH$, there are only $\kappa_n^{++}$ ultrafilters on $\kappa_n$ and if $k \geq 2$ they are all definable in $\A_{n,k}$. We can therefore apply Lemma \ref{lem-Observationsgoodness} to statements which involve ultrafilters and their Rudin-Kiesler projections. For example, 
	if $\gamma$ is a good ordinal and $\delta \in [\gamma,\lambda_n)$ satisfies that $E_n(\delta) = U$ for some ultrafilter $U$ (which must belong to $M_{n,k}(\gamma)$) then for every $\gamma' < \gamma$
	there exists some $\delta' \in (\gamma',\gamma)$ such that $E_n(\delta') = U$ as well. This ability to move around indices of $E_n$ measures without changing their essential ultrafilter information, plays a major role in the proof that the extenders-based Prikry-type poset $\po$ satisfies ${\kappa_\omega^{++}}$.c.c.

	\subsection{The forcing $(\po,\leq,\leq^*)$}
	Let $\kappa_\omega = \cup_n \kappa_n$.
	Before we proceed to define the main poset $\po$, we introduce some relevant terminology involving partial functions from ${\kappa_\omega^{++}}$ to $\lambda_n$ and subsets of $\kappa_n$.
	
	\begin{definition}[Relevant components]\label{def-relevant}
		${}$
		\begin{enumerate}
			\item A set $r_n \in [\lambda_n]^{<\kappa_n}$ is called $k$-relevant for some $k \leq n$ if it consists of $k$-good ordinals and has a maximal ordinal in the $\leq_{E_n}$ ordering. 
			\item A pair $(r_n,A_n)$ of a sets $r_n \in [\lambda_n]^{<\kappa_n}$ and $A_n \subseteq \kappa_n$ is $k$-relevant if $r_n$ is $k$-relevant  with a maximal ordinal $\gamma_n = \max(r_n)$, $A_n \in E_n(\gamma_n)$, and the following conditions hold.
			\begin{itemize}
				\item For every two ordinals $\alpha < \beta$ in $r_n$ and $\nu \in A_n$, $\pi_{\gamma_n,\alpha}(\nu) < \pi_{\gamma_n,\beta}(\nu)$. 
				\item Suppose that $\alpha \leq_{E_n} \beta \leq_{E_n} \gamma$ are three ordinals in $r_n$ and $\nu \in \pi_{\gamma_n,\gamma}``A$. Then
				\[ \pi_{\gamma,\alpha}(\nu)   =   \pi_{\beta,\alpha}\circ \pi_{\gamma,\beta}(\nu) . \]
			\end{itemize}
			\item A pair $(a_n,A_n)$ of a partial function $a_n : \kappa_{\omega}^{++} \to \lambda_n$ and a subset $A_n \subseteq \kappa_n$, is called $k$-relevant if $a_n$ is order preserving and $(\rng(a_n),A_n)$ is $k$-relevant in the above sense. 
		\end{enumerate}
	\end{definition}

	We turn to define the forcing $\po$ which adds $\kappa_{\omega}^{++}$-many new $\omega$-sequences below $\kappa_{\omega}$. 
	Conditions in $\po$ are sequences $p = \la p_n \mid n < \omega\ra$ which satisfy the following conditions:
	\begin{enumerate}
		\item There exists some $\ell < \omega$ such that for every $n < \ell$, $p_n = f_n$ is a partial function from $\kappa_{\omega}^{++}$ to $\kappa_n$, of size $|f_n| \leq \kappa_\omega$. 
		\item For every $n \geq \ell$, $p_n = \la a_n,A_n,f_n\ra$ where 
		\begin{itemize}
			\item $f_n$ is a partial function from ${\kappa_\omega^{++}}$ to $\kappa_n$ of size $|f_n| \leq \kappa_\omega$,
			\item $(a_n,A_n)$ is a $k_n$-relevant pair for some $k_n \geq 2$, where $a_n$ is a partial function from ${\kappa_\omega^{++}}$ to $\lambda_n$, and $\dom(a_n) \cap \dom(f_n) = \emptyset$.   
		\end{itemize}
		\item $\dom(a_n) \subseteq \dom(a_m)$ for every $n \leq m$.
		\item $\kappa_n \in \rng(a_n)$ for all $n \geq \ell$.
		\item The sequence $\la k_n \mid n < \omega\ra$ is nondecreasing and unbounded in $\omega$.
	\end{enumerate}
	
	We will frequently use the following conventions when referring to conditions $p \in \po$: The integer $\ell$ in the definition of $p$ will be denoted by $\ell^p$. 
	The functions $f_n$ in the definition will be denoted by $f_n^p$, and similarly, for every $n \geq \ell^p$, we will denote $a_n$ and $A_n$ by $a_n^p$ and $A_n^p$ respectively.

	The order relation $\leq$ of the poset $\po$ is the closure of the following two basic operations. 
	\begin{enumerate}
		\item Given a condition $p \in \po$, a \textbf{direct extension} of $p$ is a condition $q$ which satisfies the following conditions:
		\begin{itemize}
			\item $\ell^q = \ell^p$;
			\item $f_n^p \subseteq f_n^q$ for all $n < \omega$;
			\item $a_n^p \subseteq a_n^q$ for all $n \geq \ell^q$; and
			\item for every $n \geq \ell^q$, if $\gamma^q_n = \max(\rng(a_n^q))$ and $\gamma^p_n = \max(\rng(a_n^p))$, then $A_n^q \subseteq \pi_{\gamma^q_n,\gamma^p_n}^{-1}(A_n^p)$. 
		\end{itemize}
		The fact that $q$ is a direct extension of $p$ is denoted by $p \leq^* q$. 
		
		\item Given a condition $p \in \po$, a \textbf{one-point extension} of $p$ is a condition ${p'}$ with the following properties:
		\begin{itemize}
			\item $\ell^{{p'}} = \ell^{p} + 1$;
			\item $p_n = {p'}_n$ for all $n \neq \ell_p$; and
			\item denoting $\max(\dom(a_{\ell^p}^p))$ by $\eta$, there exists some $\nu\in A_{\ell^p}^p$ such that
			\[{p'}_{\ell^p} = f^p_{\ell^p} \cup \{ \la \tau, \pi_{a^p_{\ell^p}(\eta),a^p_{\ell^p}(\tau)}(\nu)\ra \mid \tau \in \dom(a^p_{\ell^p})\}\]
		\end{itemize}
		The fact that ${p'}$ is obtained as a one-point extension of $p$ by $\nu \in A^p_{\ell^p}$ is denoted by writing ${p'} = p \fr \la \nu\ra$. 
	\end{enumerate}
	As mentioned above, the order $\leq$ of $\po$ is the one which is generated by the two given operations.
	Therefore, for two conditions $p,q \in \po$, $q$ extends $p$ (denoted $p \leq q$) if it obtained from $p$ by finitely many applications of one-point extensions and direct extensions. 
	It is routine to verify that if $q$ extends $p$ then $q$ is a direct extension of a condition of the form 
	\[p \fr \la \nu_{\ell^p}, \nu_{\ell^p+1}, \dots, \nu_t\ra  = 
	(\dots((p \fr \la \nu_{\ell^p})\ra \fr \la\nu_{\ell^p+1})\ra \dots ) \fr \la\nu_t\ra \] 
	which is the condition obtained from $p$ by taking $(t+1-\ell_p)$ many one-point extensions with ordinals $\nu_n  \in A^p_n$ for every $n$, $\ell^p \leq n \leq t$.
	
	Let $p = \la p_n \mid n < \omega\ra$ be a condition in $\po$. For every $m < \omega$ we decompose $p$ into the two parts, 
	$p\uhr m = \la p_n \mid n < m\ra$ and $p\dhr m = \la p_n \mid n \geq m\ra$. 
	With this, we define $\po_{<  m} = \{ p\uhr m \mid p \in \po\}$ and $\po_{\geq m} = \{ p\dhr m \mid p \in \po\}$. 
	It is not difficult to see that that the orders $\leq$ and $\leq^*$ on $\po$ naturally order relations on $\po_{<  m}$ and $\po_{\geq m}$ for every $m < \omega$. Moreover, for every $p \in \po$ and $m \leq \ell^p$, the poset $(\po/p,\leq)$ naturally breaks into the product $(\po_{\leq m}/p\uhr m, \leq) \times (\po_{>m}/ p\dhr m, \leq)$. The same holds if we replace $\leq$ by $\leq^*$. Finally, we note that if $m \leq \ell^p$ then 
	the restrictions of $\leq$ and $\leq^*$ to $\po_{<  m}/ p\uhr m$ coincide. 
	
	We list several basic properties of $\po$ which are immediate consequences of the definitions.
	\begin{lemma}\label{Lem-PObasicproperties}${}$
		\begin{enumerate}
			\item $\po$ satisfies the Prikry condition. That is, for every statement $\sigma$ of the forcing language $(\po,\leq)$ and every condition $p \in \po$ there exists a direct extension $p^* \geq^* p$ such that $p^*$ decides $\sigma$. The same is true for $\po_{<  m}$ and $\po_{\geq m}$ for every $m < \omega$.
			\item For every $m < \omega$, the direct extension order $\leq^*$ of $\po_{\geq m}$ is $\kappa_m$-closed.
			\item For every condition $p \in \po$ and $m \leq \ell^p$, the order $\leq$ of $\po_{<  m}$ is $\kappa_{\omega}^+$-closed. 
			\item For every condition $p \in \po$, the direct extension order of $\po/p$ is $\kappa_{\ell^p}$-closed.
		\end{enumerate}
	\end{lemma}
	
	The last property implies that the forcing $\po$ does not add new bounded subsets to $\kappa_{\omega}$.
	Next, we state a technical strengthening of the Prikry Lemma which follows from the argument of its proof. 
	\begin{lemma}\label{Lem-meetdense}
		Let $D \subseteq \po$ be an open dense set (in the usual order $\leq$). For every condition $p \in \po$ there are $k \geq \ell^p$ and $p^* \geq^* p$ so that for every $\vec{\nu} = \la \nu_{\ell^p},\dots, \nu_{k-1}\ra \in \prod_{\ell^p \leq n < k} A_n^{p^*}$, $p^* \fr \vec{\nu}$ belongs to $D$ \footnote{Note that when $k = \ell^p$, the product of the sets $A_n$ is empty, and therefore $p^* \in D$.}.
	\end{lemma}
	
	A standard application of Lemma \ref{Lem-meetdense} it that the forcing $\po$  preserves $\kappa_\omega^+$. We sketch the argument.
	\begin{corollary}
		$\po$ does not collapse $\kappa_\omega^+$.
	\end{corollary}
	\begin{proof}[Proof Sketch.]
		The fact that $\kappa_{\omega}$ is singular in $V$ implies that if $\kappa_\omega^+$ is collapsed then $\po$ introduces a cofinal function $f : \rho \to \kappa_{\omega}^+$ from some $\rho < \kappa_{\omega}$. Let $\name{f}$ be a $\po$-name for a function from $\rho$ to $\kappa_{\omega}^+$, and $p$ be a condition $\po$ with $\kappa_{\ell^p} > \rho$. For every $i < \rho$, let $D_i$ be the dense open subset of $\po$ of conditions $q \in \po$ which decide the ordinal value of $\name{f}(\can{i})$. Since the direct extension order of $\po/p$ is $\kappa_{\ell^p}$-closed, we can repeatedly use Lemma \ref{Lem-meetdense} and construct a $\leq^*$-increasing sequence of conditions $\la p^i \mid i \leq \rho\ra$ such that for every $i < \rho$ there exists some $n_i \geq \ell^p$ so that $p^i \fr \vec{\nu}$ belongs to $D_i$ for all $\vec{\nu}\in \prod_{\ell^p\leq n < n_i}A^{p^i}_n$. 
		Let $p^* = p^\rho$. It follows that there are functions $F_i$, $i < \rho$ with $F_i : \prod_{\ell^p \leq n < n_i} A^{p^*}_n \to \kappa_\omega^+$ for all $i$, 
		such that for each $i < \rho$ and $\vec{\nu} \in \prod_{\ell^p \leq n < n_i} A^{p^*}_n$, $p^* \fr \vec{\nu} \force \name{f}(\can{i}) = \can{F_i}(\vec{\nu})$.
		It follows that $p^*$ forces that $\rng{\name{f}}$ is a subset of $X = \bigcup_{i<\gamma}\rng(F_i)$, which has size $|X| \leq \kappa_{\omega}$. Consequently, $p^*$ forces that $\name{f}$ is bounded in $\kappa_\omega^+$.
	\end{proof}

	\subsection{The essential generic information}
	Let $G \subseteq \po$ be a generic filter and denote ${\kappa_\omega^{++}}^V$ by $\lambda$. Without loss of generality, we assume $G$ contains a condition $p$ with $\ell^p = 0$. 
	A standard density argument shows that for every $\alpha<\lambda$ and $n < \omega$ there is a condition $p \in G$ with $\ell^p > n$, so that $\alpha \in \dom(f_n^p)$\footnote{note that if $\alpha \in \dom(a_n^p)$ then $\alpha \in \dom(f_n^q)$ for every extension $q$ of $p$ which involves at least $n$ one-point extensions.}. 
	It is easy to see that the value $f_n^p(\alpha) < \kappa_n$ does not depend on the choice of the condition $p \in G$, and we denote it by $t_\alpha(n)$. It follows that  $t_\alpha \in \prod_n \kappa_n$.
	Also, recall that by our definition of conditions $p \in \po$, $\kappa_n \in \rng(a_n^p)$ for some $p \in G$. Let $\alpha^0_n  < {\kappa_\omega^{++}}$ be the unique value for which 
	$\kappa_n  = a_n^p(\alpha^0_n)$ and define $\rho_n = t_{\alpha^0_n}(n)$. The sequence $\vec{\rho} = \la \rho_n \mid n < \omega\ra$ is generic for the diagonal Prikry forcing (\cite{Gitik-HB}) by the normal measures $\la E_n(\kappa_n) \mid n < \omega\ra$. 
	
	For the proof of Theorem \ref{thm1}, we will only care about functions $t_\alpha$ which originate in the ``extender components`` of $G$, namely, for values $\alpha$ which belong to $\dom(a_m^p)$ for some $p \in G$ (and thus, also to $\dom(a_n^p)$ for every $n \geq m$). The following definition makes this notion precise.
	\begin{definition}
		We define the set  $\AP_G \subseteq \lambda$ of active points in $V[G]$ by $\AP_G = \{ \alpha < \lambda \mid \alpha \in \dom(a_n^p) \text{ for some } p \in G \text{ and } n < \omega\}$. 
	\end{definition}
	
	A simple density argument shows that the set $\AP_G$ is unbounded in $\lambda$. 
	Let $\vec{t} = \la t_\alpha \mid \alpha \in \AP_G\ra$. It is easy to see $\vec{t}$ is increasing in the ordering $<^*$.
	Like most extender-based forcings, it is typical that $\vec{t}$ is forms a scale 
	in a product $\prod_n \tau_n$ of cardinals $\tau_n > \rho_n$ such that, loosely speaking, each $\tau_n$ is to $\rho_n$ what $\lambda_n$ is to $\kappa_n$. An example of such a result involving different extender-based posets can be found in \cite{Gitik-HB}. For an argument which involves short extenders forcings, we refer the reader to \cite{Gitik-EBF1}. We state two relevant results.
	\begin{lemma}\label{lem-genericscaleproduct}
		${}$
		\begin{enumerate}
			\item Suppose that for each $n < \omega$, $\lambda_n = j_n(h_n)(\kappa_n)$ for some function $h_n : \kappa_n \to \kappa_n$. Then in $V[G]$, $\vec{t}$ is a scale on the product $\prod_n h_n(\rho_n)$. For example if $\lambda_n = \kappa_n^{+n+2}$ then $\vec{t}$ is cofinal in $\prod_n \rho_n^{+n+2}$.
			\item If $\lambda_n = j_n(\kappa_n)$ for each $n <\omega$, then $\vec{t}$ is a scale on $\prod_n \kappa_n$. 
		\end{enumerate}
	\end{lemma}
	
	As will be shown below, the generic sequence $\vec{t}$ has some appealing properties which fit the results established in Section \ref{section-pre}.
	Two apparent issues need to be taken care of before we can apply the results of Section \ref{section-pre} to $\vec{t}$ in $V[G]$. 
	The first one is that the indices of the sequence $\vec{t}$ are not all the ordinals below $\lambda$, but only an unbounded subset. This issue is merely cosmetic, and it is straightforward to verify that all the results of Section \ref{section-pre} apply to sequences $\vec{f}$ with domain $A \subseteq \lambda$, as long as we restrict the argument to domain points $\delta \in A$ (For example, the statement of Proposition \ref{prop-IAmain} applies to all points $\delta \in S_{\vec{a}} \cap C \cap A$ which are continuity points of $\vec{f}$). 
	The second issue, which is much more substantial and demands a revision of the forcing $(\po,\leq)$ is that $\lambda = {\kappa_\omega^{++}}^V$ need not be a cardinal in $V[G]$. Indeed, the forcing $\po$ fails to preserve ${\kappa_\omega^{++}}^V$ and does not generate a model in which $\SCH$ fails.
	
	Gitik resolved this by identifying a quotient order of $(\po,\leq)$, introduced by an equivalence relation $\lra$ on $\po$, which satisfies ${\kappa_\omega^{++}}$.c.c but does not affect the essential generic information $\vec{t}$. Namely, every two conditions $p,{p'}$ which are $\lra$ equivalent force the exact same statments about $\vec{t}$.
	We proceed to review the details. 

	\subsection{The order $\rar$}
	Fix integers $1 < k \leq n$, and let $\Lan_{n,k}$ be the language of the structure $\A_{n,k}$. 
	We define the $(n,k)$-type of an element $x \in \A_{n,k}$ to be the $\Lan_{n,k}$-type which is realized by $x$ in the model $\A_{n,k}$. We denote the type by $\tp_{n,k}(x)$ and identify it with a subset of $\kappa_n^{+k} = |\Lan_{n,k}|$.
	We will also need a relativized version of these types. 
	For every element $r \in \A_{n,k}$ let $\A^r_{n,k}$ be the model of the expanded language $\Lan_{n,k}^c$ in which a new constant symbol $c$ is interpreted as $r$, and 
	define the $(n,k)$-type $x \in \A_{n,k}$ relative to $r$ to be the $\Lan_{n,k}^c$-type realized by $x$ in the model $\A^r_{n,k}$. We denote the $r$-relativized type by $\tp^r_{n,k}(x)$. 
	
	Since we assume $V$ satisfies the $\GCH$, for each $n < \omega$ there are only $\kappa_n^+$ many functions $\pi : \kappa_n \to \kappa_n$, and only $\kappa_n^{++}$ many ultrafilters $U$ on $\kappa_n$. Therefore, if $k \geq 2$ every such function $\pi$ and ultrafilter $U$ are definable in the language of $\A_{n,k}$ which contain constants for every $\tau < \kappa_n^{++}$. The following is an immediate consequence. 
	\begin{lemma}\label{lemma-typequiv}
		Fix $n< \omega$ and $k \geq 2$. Let $x$ be a set in $[\lambda_n]^{<\kappa_n}$. The following features of $x$ are completely determined by its type $\tp_{n,k}(x)$:
		\begin{enumerate}
			\item $\otp(x) < \kappa_n$;
			\item the ultrafilter $E_n(\alpha)$ for every $\alpha \in x$;
			\item the projection maps $\pi_{\beta,\alpha}$ for every two ordinals $\alpha,\beta \in x$ with $\beta \geq_{E_n} \alpha$.
		\end{enumerate}
		Similarly, the relative type $\tp_{n,k}^r(x)$ determines the same for $x\cup r$ because it determines the type $\tp_{n,k}(r \cup x)$.
	\end{lemma}

	\begin{definition}\label{def-equiv}
		${}$
		\begin{enumerate}
			\item     Fix $n < \omega$ and let $r,r'$ be two sets in $[\lambda_n]^{<\kappa_n}$ for some $n < \omega$.
			We say that $r,r'$ are $k$-equivalent if $\tp_{n,k}(r) = \tp_{n,k}(r')$. 
			\item  Let $p_n = \la a_n,A_n,f_n\ra$ and ${p'}_n = \la a'_n,A'_n,f'_n\ra$ be two $k$-relevant components for some $k < \omega$. We write $p_n \iff_{n,k} {p'}_n$ if and 
			only if $\rng(a_n)$ and $\rng(a_n')$ are $k$-equivalent sets in $[\lambda_n]^{<\kappa_n}$, 
			$A_n = A_n'$, and $f_n = f'_n$.
			\item For every two conditions $p,{p'} \in \po$, we write $p \iff {p'}$ if and only if $\ell^p = \ell^{{p'}}$ and there is a nondecreasing unbounded sequence $\la k_n^* \mid n < \omega\ra$ of integers $k_n^* \geq 2$ such that for ${p'}_n = p_n$  for every $n < \ell^p$, and $p_n \iff_{n,k^*_n} {p'}_n$ for every $n \geq \ell^p$.
		\end{enumerate}
	\end{definition}
	
	It is straightforward to verify that $\iff$ is an equivalence relation. 
	We also note that if $r,r'\in [\lambda_n]^{<\kappa_n}$ are $k$-equivalent then they are $l$-equivalent for every $l < k$. Therefore, if $p_n \iff_{n,k} {p'}_n$ then $p_n \iff_{n,l} {p'}_n$.
	
	\begin{definition}
		Let $p,q$ be two conditions of $\po$. We write $p \rar q$ to mean that $q$ is obtained from $p$ by finitely many  $\geq-$extensions and $\iff$ transitions. 
	\end{definition}
	Therefore if $p \rar q$ then $q$ is stronger (more informative) than $p$. 
	It is clear that every two conditions $p\iff {p'}$ in $\po$ are forcing equivalent in the poset $(\po,\rar)$, and 
	by Lemma \ref{lemma-typequiv}, 
	that they force the exact same statemets about $\vec{t}$. 
	
	The following two results are crucial to the success of the forcing construction. 
	\begin{theorem}[Gitik, see \cite{GitUng-SEF}]\label{thm-SEFchain}
		${}$
		\begin{enumerate}
			\item If $p \iff {p'}$ are two equivalent conditions and $q'$ extends ${p'}$ in $\leq$, then there are conditions $q'' \geq {q'}$ and $p'' \geq {p'}$ such that $p'' \iff q''$. Consequently, for every dense open set $D$ in the poset $(\po,\rar)$ and a condition $p \in \po$ there exists some $p'' \geq p$ in $D$.
			
			\item $(\po,\rar)$ satisfies ${\kappa_\omega^{++}}$.c.c.
		\end{enumerate}
	\end{theorem}
	
	We note that the first statement of Theorem \ref{thm-SEFchain} implies that the identity function forms a forcing projection of $(\po,\leq)$ onto $(\po,\rar)$ and therefore allows us to use the Prikry forcing machinery of $(\po,\leq)$ to analyze $(\po,\rar)$. In particular, $(\po,\rar)$ does not introduce new bounded subsets to $\kappa_{\omega}$ and does not collapse $\kappa_{\omega}^+$. The second statement asserts that $(\po,\rar)$ does not collapse cardinals $\lambda \geq {\kappa_\omega^{++}}$ and allows us to apply the results of Section \ref{section-pre} to the generic scale $\vec{t}$.\\

	We sketch the argument for ${\kappa_\omega^{++}}$.c.c to justify Definitions \ref{def-relevant}, \ref{def-goodpoints}, and the use of $k$-good ordinals. 
	Suppose that $\{p_\alpha \mid \alpha < {\kappa_\omega^{++}}\}$ is a family of conditions of $\po$. By applying standard $\Delta$-system and pressing down arguments, it is possible to find a subfamily of the same size such that for every two conditions in the subfamily, $p_\alpha,p_\beta$ with $\alpha < \beta$, they agree on $\ell^{p_\alpha} = \ell^{p_\beta} = \ell$, and  the following hold for each $n< \omega$:
	\begin{enumerate}
		\item $f^{p_\alpha}_n$ and $f^{p_\beta}_n$ are compatible functions (i.e., they agree on the values of common domain ordinals);
		\item $A_n^{p_\alpha} = A_n^{p_\beta} = A_n$ and $\rng(a_n^{p_\alpha}) = \rng(a_n^{p_\beta}) = r_n$ for all $n \geq \ell$; 
		\item $\dom(a_n^{p_\alpha}) \cap \alpha = \dom(a_n^{p_\beta}) \cap \beta = d_n$ for some $d_n \in [{\kappa_\omega^{++}}]^{<\kappa_n}$; 
		\item $\dom(a_n^{p_\alpha})\setminus \alpha \subseteq \beta$; and 
		\item $k_n^{p_\alpha} = k_n^{p_\beta} = k_n$.
	\end{enumerate} 
	The only obstruction to $p_\alpha$ and $p_\beta$ having a common extension is that the disjoint sets $\dom(a_n^{p_\alpha}) \setminus \alpha$ 
	and $\dom(a_n^{p_\beta}) \setminus \beta$ are mapped by the order preserving functions $a_n^{p_\alpha}$ and $a_n^{p_\beta}$, respectively, to the same 
	ordinals in $r_n$.
	This makes it impossible for $a_n^{p_\alpha} \cup a_n^{p_\beta}$ to be order preserving.
	To circumvent this, we use the equivalence relation $\iff$ to replace $p_\alpha$ with an equivalent $p_\alpha'$ so that $a_n^{p_\alpha'}$ is compatible with $a_n^{p_\beta}$. Let $x = a_n^{p_\alpha}``(\dom(a_n^{p_\alpha}) \setminus \alpha) = a_n^{p_\beta}``(\dom(a_n^{p_\beta}) \setminus \beta)$, $\gamma = \min(x)$ and $\gamma' = \sup(r_n \setminus x)$. Recall that since $\gamma$ is $k_n$-good there is a substructure $M_{n,k_n}(\gamma) \elem \A_{n,k_n}$ such that $M_{n,k_n}(\gamma) \cap \lambda_n = \gamma$. The language $\Lan_{n,k}$ includes a constant for each $\tau < \kappa_{n}^{+k_n}$. Therefore $M_{n,k_n}(\gamma)$ contains all $(n,k_{n-1})$-types and in particular the relative type $t = t^{r_n}_{n,k_n-1}(x)$. By Lemma \ref{lem-Observationsgoodness} there exists a set of ordinals $x' \subseteq \gamma \setminus (\gamma'+1)$ which realizes the same type $t$. This, and Lemma \ref{lemma-typequiv} in turn, imply that $x'$ consists of $k_{n-1}$-ordinals and that 
	$\tp_{n,k_n-1}(r_n \cup x') = \tp_{n,k_n-1}(r_n \cup x)$.  Let $a_n'$ be the partial and order preserving function obtained from $a_n^{p_\alpha}$ by replacing 
	the range $r_n \cup x$ with $r_n \cup x'$. By our choice of $x'$ we have that  $(a_n', A_n^{p_\alpha})$ is $(k_{n}-1)$-relevant.
	If $p' = \la p'_n \mid n < \omega\ra$ is the sequence obtained from $p_\alpha$ by defining $a_n^{p'} = a_n'$ and $A_n^{p'} = A_n^{p_\alpha}$,  then $p'$ is a condition in $\po$ 
	which is $\iff$ equivalent to $p_\alpha$. Finally, it is clear from the construction that $a_n^{p'} \cup a_n^{p_\beta}$ is order preserving and $(k_{n}-1)$-relevant. We conclude that $p_\beta$ and $p' \iff p_\alpha$ are compatible in $\leq$ and thus $p_\beta$ and $p_\alpha$ are compatible in $\rar$.

	
	\subsection{Proof of Theorem \ref{thm1}}
	The last argument justifies the restriction in the definition of conditions in $\po$ to $k$-good ordinals. This restriction is mild since the set of $n$-good ordinals is closed unbounded in  $\kappa_n$, which leaves plenty of room to choose extender indices from  $E_n$ to construct the generic scale $\vec{t}$. 
	Our situation requires more caution, as would like to control the extender indices $\gamma \in \rng(a_n^p)$
	to the level where we can guarantee that $\gamma \in j_n(S_n)$ for a prescribed stationary subset $S_n$ of $\kappa_n$. By the elementarity of $j_n$, it is clear that the set $T_n = j_n(S_n)$ is stationary in the codomain of $j_n$. 
	However, $T_n$ need not be stationary in $V$ and thus might not contain good ordinals.
	It is for this reason that we require that $j_n$ possess a stronger (large cardinal) property than the one presented by $E_n$. For example, while requiring that each $j_n$ and $E_n$ are superstrong suffices for obtaining a generic scale on $\prod_n \kappa_n$, we will further assume each $j_n$ is $(+1)$-extendible; a property which is not reflected in its derived extender $E_n$. 
	We proceed to the proof of Theorem \ref{thm1}. 
	
	Suppose that $\la \kappa_n \mid n < \omega\ra$ is an increasing sequence of $(+1)$-extendible cardinals in a model $V$ of $\GCH$. For each $n < \omega$, let $j_n : V_{\kappa_n+1} \to V_{\lambda_n+1}$ be a $(+1)$-extendible embedding (i.e., $\lambda_n = j_n(\kappa_n)$) and $E_n$ be the $(\kappa_n,\lambda_n)$ extender derived from $j_n$. 
	Denote ${\kappa_\omega^{++}}$ by $\lambda$. By Theorem \ref{thm-ShelahIA}, for every regular uncountable cardinal $\mu < \kappa_{\omega}$ there exists a sequence $\vec{a}^\mu = \la a^\mu_\alpha \mid \alpha < \lambda\ra$ of bounded subsets of $\lambda$ such that $S_{\vec{a}^\mu}^V \cap \cof(\mu)$ is stationary in $\lambda$.
	We force over $V$ with the short extenders poset $(\po,\rar)$ defined by the extenders $\la E_n \mid n < \omega\ra$.
	By Theorem \ref{thm-SEFchain}, $(\po,\rar)$ satisfies $\lambda$.c.c and therefore $S(\vec{a}^\mu)^V \cap \cof(\mu)$ remains stationary in $\lambda$ for all regular uncountable $\mu <\kappa_{\omega}$.
	
	\begin{remark}
		It is clear from Defintion \ref{def-IApoints}, that if $\gamma$ is an approachable ordinal with respect to $\vec{a}^\mu$ in $V$, then it is such in every generic extension $V[G]$.
		On its face, $V[G]$ can contain new ordinals which are approachable with respect to a sequence $\vec{a}^\mu$, however, using Lemma \ref{Lem-meetdense}, it is possible to show that $S_{\vec{a}}^V = S_{\vec{a}}^{V[G]}$. 
		The last fact will not be used in the proof of Theorem \ref{thm1} below, which only requires that the set $S_{\vec{a}}^V \cap \cof(\mu)$ is stationary in $V[G]$ and contains ordinals which are approachable with respect to $\vec{a}^\mu$. 
	\end{remark}
	
	By Lemma \ref{lem-genericscaleproduct}, $G$ introduces a scale $\vec{t} = \la t_\alpha \mid \alpha \in \AP_G\ra$ in the product $\prod_n \kappa_n$.
	Fix a regular uncountable cardinal $\mu < \kappa_\omega$, and suppose that $m < \omega$ is the first integer such that $\mu < \kappa_m$, and
	$\vec{S} = \la S_n \mid m \leq n < \omega\ra$ is a sequence of stationary sets $S_n \subseteq \kappa_n \cap \cof(\mu)$, in $V$.
	We claim that $\vec{S}$ is tightly stationary in $V[G]$. It is sufficient to show that for every algebra $\A$ which expands $\la H_\theta^{V[G]}, \in ,<_\theta, \vec{t}, \vec{a}^\mu\ra$ there is a tight substructure $M \elem \A$ so that $\sup(M \cap \kappa_n) \in S_n$ for almost all $n < \omega$. 
	Moreover, Proposition \ref{prop-IAmain} guarantees that in $V[G]$, for every algebra $\A$ which expands  $\la H_\theta, \in,<_\theta,\vec{t},\vec{a}^\mu\ra$ for some regular cardinal $\theta > \lambda$ there is a closed unbounded set $C \subseteq \lambda$ with the property that for every $\delta \in S_{\vec{a}}^{V[G]} \cap C$, 
	if $\delta$ is a continuity point of $\vec{t}$ then there is a tight substructure $M \elem \A$ such that $\sup(M \cap \kappa_n) = t_\delta(n)$ for almost all $n < \omega$. 
	It is therefore sufficient to verify that $\vec{t}$ satisfies the following property.
	\begin{proposition}\label{proposition-mainthm1}
		For every closed unbounded subset $C \subseteq \lambda$ there exists an ordinal $\delta \in C \cap S_{\vec{a}} \cap \cof(\mu)$ which is a continuity point of $\vec{t}$ and 
		$t_\delta(n) \in S_n$ for almost all $n < \omega$.
	\end{proposition}
	\noindent\emph{proof (Proposition \ref{proposition-mainthm1}).}\\
	Since $(\po,\rar)$ satisfies $\lambda$.c.c, every closed unbounded subset of $\lambda$ in $V[G]$ contains a closed unbounded set in $V$. 
	It is therefore sufficient to provide a density argument and show that for every closed unbounded set $C \subseteq \lambda$ in $V$ and a condition $p \in \po$,
	there are $\delta \in S_{\vec{a}} \cap C \cap \cof(\mu)$ and an extension $p^*$ of $p$ which forces that $\delta$ is a continuity point of $\vec{t}$ and that $t_\delta(n) \in S_n$ for all $n \geq \max(m,\ell^p)$.
	To this end, fix a condition $p \in \po$ and a club $C \subseteq \lambda$. We may assume that $\ell^p \geq m$. Let $a = \bigcup_n (\dom(a_n^p) \cup \dom(f_n^p)) \in [\lambda]^{\kappa_\omega}$. We can pick some $\delta \in S_{\vec{a}} \cap C \cap \cof(\mu)$ which is strictly above $\sup(a)$, and a continuous, increasing, and cofinal sequence $d  = \la \delta(i) \mid i < \cf(\delta)\ra$ in $\delta \setminus (\sup(a)+1)$. 
	For each $n \geq \ell^p$ let $T_n = j_n(S_n) \subseteq \lambda_n$. By the elementarity of $j_n$, $T_n$ is a stationary subset of $\lambda_n$ in $V_{\lambda_n+1}$, and thus also stationary in $V$. Furthermore, the fact $\mu < \kappa_n$ implies that $T_n \subseteq \cof(\mu)$. It follows that $T_n$ contains an $n$-good ordinal $\delta_n > \sup(\rng(a_n^p))$ of cofinality $\mu$ which is also a limit of an increasing continuous sequence of $n$-good ordinals, $d_n = \la \delta(i)_n \mid i < \mu\ra$.
	Extend the partial function $a_n^p$ to a function $a_n'$ which is defined by $a_n' = a_n^p \cup \{ \la \delta,\delta_n\ra\} \cup \{\la \delta(i),\delta_n(i)\ra \mid i < \mu\}$. 
	Next, we choose an ordinal $\rho \in \lambda \setminus (\delta+1)$, and for each $n \geq \ell^p$, pick an $n$-good ordinal $\rho_n > \delta_n$ which is an $\leq_{E_n}$-upper bound for $\rng(a_n')$ (recall that the order $\leq_{E_n}$ is $\kappa_n$-directed).
	Define $a_n^* = a_n' \cup \{ \la \rho,\rho_n\ra\}$, and let $A_n^* \subseteq \pi^{-1}_{\rho_n,\max(\rng(a^p_n))}(A_n^p)$ be the set of all ordinals $\nu$ which satisfy the following two conditions:
	\begin{enumerate}
		\item $\pi_{\rho_n,\delta_n}(\nu) \in S_n$; and
		\item $\la \pi_{\rho_n,\delta_n(i)}(\nu) \mid i < \mu\ra$ is increasing, continuous, and confinal in $\pi_{\rho_n,\delta_n}(\nu)$.
	\end{enumerate}
	Finally, let $p^* = p^* = \la p_n^* \mid n < \omega\ra$ be defined by 
	\[ 
	p_n^* = 
	\begin{cases}
	p_n &\mbox{ if }  n < \ell^p \\
	\la a_n^*,A_n^*,f_n^p\ra &\mbox{ if }  n \geq \ell^p 
	\end{cases}
	\]
	It is straightforward to verify $p^*$ is a direct extension of $p$ in $\po$, and that 
	\[p^* \force \name{t_\delta} \text{ is an eub of } \name{\vec{t}}\uhr d \text{ and } \name{t_\delta(n)} \in \can{S_n} \text{ for all } n \geq \ell^p  \]
	The fact $d= \la \delta(i) \mid i < \cf(\delta)\ra$ is cofinal in $\delta$ implies that $\vec{t}\uhr\delta$ is cofinally interleaved with $\vec{t}\uhr d$. Hence $p^*$ forces that $\name{t_\delta}$ is an eub of $\vec{t}\uhr\delta$, and thus that $\delta$ is a continuity point of $\vec{t}$. 
	\qed{Proposition \ref{proposition-mainthm1}}\\
	\qed{Theorem \ref{thm1}}

	\section{Down to $\aleph_\omega$}\label{section-thm2}
	In this section we prove Theorem \ref{thm2} which is similar to Theorem \ref{thm1} with two major differences. 
	\begin{enumerate}
		\item The sequence of regular cardinals to which the result applies is $\la \omega_{s_n} \mid n < \omega\ra$ for some subsequence $\la s_n \mid n < \omega\raggedbottom$ of $\omega$. 
		\item This sequence is Prikry geneneric over a ground model and therfore does not exist in the core model or the mental.
	\end{enumerate}

	The last property allows us to reduce the large cardinal assumption from the level of extendibility to the hypermeasurability assumption of an increasing sequence $\la \kappa_n \mid n < \omega\ra$ such that each $\kappa_n$ is $\kappa_n^{+n+3}$ strong. Fix for each $n$ a $\kappa_n^{+n+3}$-strong emedding $j_n : V_{\kappa_n+1} \to N_n$ with critical point $\kappa_n$, and let $E_n$ be the $(\kappa_n,\kappa_n^{+n+2})$-extender derived from $j_n$. The gap between the strength of the extender $E_n$ (which is $\kappa_n^{+n+2}$) to the strength of the embedding $j_n$ ($\kappa_n^{+n+3}$) is analogous to the gap between the superstrong extenders $E_n$ and the $(+1)$-extendible embeddings $j_n$ in the proof of Theorem \ref{thm1}. It will be used to insure a name of a stationary subset $T_n$ of $\kappa_n^{+n+2}$ in $N_n$, is also a name for a stationary set in $V$, and thus must contain many good points $\delta < \kappa_n^{+n+2}$ in the sense of Definition \ref{def-goodpoints}.
	
	To prove Theorem \ref{thm2} we will modify the short extenders forcing $\po$ from the previous section. 
	A key feature of the revised version of $\po$ is that it subsumes a ``vanilla`` diagonal Prikry forcing with interleaved collapses, which will be denoted here by $\bar{\po}$.
	The forcing $\bar{\po}$ introduces a single Prikry sequece $\vec{\rho} = \la \rho_n \mid n < \omega\ra$ which is associated with the sequence of normal measures $\la E_n(\kappa_n) \mid n < \omega\ra$. Besides adding the diagonal Prikry sequence $\vec{\rho}$, the poset $\bar{\po}$ incorporates Levy posets which further collapse the cardinals in the intervals $(\rho_n^{+n+3},\kappa_n)$ and $(\kappa_n^{+n+3},\rho_{n+1})$ for every $n$. 
	Therefore, in a $\bar{\po}$ generic extension $V[\bar{G}]$, the sequence of cardinals $\la \rho_n^{+n+2} \mid n < \omega\ra$ forms a subsequence $\la \omega_{s_n} \mid n < \omega\ra$ of the $\omega_n$s.
	$V[\bar{G}]$ will be the ground model that is specified in the statement of Theorem \ref{thm2}. We will argue that every fixed-cofinality sequence $\vec{S} = \la S_n \mid n < \omega\ra$ of stationary sets $S_n \subseteq \rho_n^{n+2}$ is tightly stationary in a further forcing extension over $V[\bar{G}]$. This will be done by proving that $\vec{S}$ is tight in the (full) $\po$ generic extension $V[G]$ which can be seen as a $\po/\bar{\po}$ forcing extension of $V[\bar{G}]$. 
	
	Let us explain why this description dictates an additional revision of $\po$ (besides adding an inverleaved collapse posets). A standard analysis of Prikry type forcings shows that $\bar{\po}$ satisfies a version of Lemma \ref{Lem-meetdense} which implies that if $\name{S_n}$ is a $\bar{\po}$-name of a subset of $\rho_n^{+n+2}$ then for every condition $q \in \bar{\po}$ there exists a direct extension $p^*$ such that every choice of the first $(n+1)$ generic Prikry points $\vec{\rho}_{n+1} = \la \rho_0,\dots,\rho_n\ra$ reduces $\name{S_n}$ to a name $\bar{S}_n(\vec{\rho}_{n+1})$ which depends only on the collapse product of cardinals below $\rho_n$. Obtaining this substitution of names brings us sufficiently close to the assumptions of Theorem \ref{thm1} and allows us to apply a similar argument, and show that there are sufficiently many good IA ordinals $\delta < {\kappa_\omega^{++}}$ that can be generically map to some $n$-good ordinal $\delta_n < \kappa_n^{+n+2}$, such that $\delta_n$ is forced to belong to the stationary name $T_n(\vec{\rho}_{n+1}) = j_n({\bar{S}_n}(\vec{\rho}_{n+1}))$ by some suitable collapse conditions.
	The caveat in this description is that the choice of $\delta_n$ assumes the knowledge of the first diagonal Prikry points $\vec{\rho}_{n+1}$.
	To circumvent this issue, we modify the construction of $\po$ by requiring that in conditions $p \in \po$, the extender indicies maps $a_n = a_n^p$ depend
	on the preceeding diagonal Prikry points $\vec{\rho}_{n} = \la \rho_0,\dots,\rho_{n-1}\ra$ below $\kappa_{n-1}$ \footnote{we will be able to avoide knowing the value of the next point $\rho_n$ by some standard integration manipulation.}.
	Namely, $a_n$ will be a function which maps every potential Prikry initial segment $\vec{\rho}_n$ to a partiral function, $a_n^{\vec{\rho}_n} : {\kappa_\omega^{++}} \to \lambda_n$, with similar properties to the functions $a_n$ which were used in the previous section.
	Accordingly, we will also make the measure one set component $A_n = A_n^p$ to depend on the same information. Therefore $A_n$ will be a function which will map every relevant $\vec{\rho}_n$ to a set $A_n^{\vec{\rho}_n} \in E_n(\max(\rng(a_n^{\vec{\rho}_n}))$. 
	
	We proceed to define $\bar{\po}$ and $\po$.
	\subsection{The poset $\bar{\po}$}
	Suppose that $V$ is a model which contains an increasing sequence $\vec{\kappa} = \la \kappa_n \mid n < \omega\ra$ of cardinals so that each $\kappa_n$ is $\kappa_{n}^{+n+3}$ strong. 
	For each $n < \omega$ we fix a $\kappa_{n}^{+n+3}$-strong embedding $j_n : V_{\kappa_n+1} \to N_n$ and let $E_n$ be the $(\kappa_n,\kappa_n^{n+2})$-extender derived from $j_n$. We denote $\kappa_n^{+n+2}$ by $\lambda_n$. For notational simplicity, we define $\kappa_{-1} = \omega_1$.
	
	The poset $\bar{\po}$ is a diagonal Prikry forcing with interleaved Levy collapse posets.
	Conditions $\bar{p} \in \bar{\po}$ are of the form $\bar{p} = \la \bar{p}_n \mid n < \omega\ra$ and satisfy the following consitions:
	\begin{enumerate}
		\item There exists some $\ell < \omega$ such that $\bar{p}_n = \la \rho_n, g_n,h_n\ra$ for every $n < \ell$, where
		$\rho_n \in (\kappa_{n-1},\kappa_n)$, $g_n \in \col(\kappa_{n-1}^{+(n-1)+3}, <\rho_n)$, and $h_n \in \col(\rho_n^{+n+3},<\kappa_n)$. 
		\item For every $n \geq \ell$, $p_n = \la \np{A_n}, g_n,H_n\ra$, where $g_n \in \col(\kappa_{n-1}^{+(n-1)+3},<\kappa_n)$, 
		$\np{A_n} \in E_n(\kappa_n)$ consist of regular cardinals $\rho$ such that $g_n \in \col(\kappa_{n-1}^{+(n-1)+3},<\rho)$, 
		and $H_n$ is a function with $\dom(H_n) = \np{A}_n$ and $H_n(\rho) \in \col(\rho^{+n+3},<\kappa_n)$ for each $\rho$ in its domain. 
	\end{enumerate}
	A usual, we denote $\ell,g_n,h_n, \np{A_n}, H_n, \rho_n$ by $\ell^{\bar{p}},g_n^{\bar{p}},h_n^{\bar{p}}, \np{A_n}^{\bar{p}}, H_n^{\bar{p}}, \rho_n^{\bar{p}}$, respectively.
	Furthermore, we denote the sequence $\la \rho_0^p,\dots, \rho_{\ell^p-1}^{\bar{p}}\ra$ by $\vec{\rho}_{\bar{p}}$.\\
	
	A condition $\bar{q} \in \bar{\po}$ is a direct extension of $\bar{p}$ (denoted $\bar{q} \geq^* \bar{p}$) if the following conditions hold:
	\begin{itemize}
		\item 
		$\ell^{\bar{q}} = \ell^{\bar{p}}$;
		
		\item for every $n < \ell^{\bar{p}}$, $g_n^{\bar{q}} \geq g_n^{\bar{p}}$ and $h_n^{\bar{q}} \geq h_n^{\bar{p}}$; 
	
		\item for every $n \geq \ell^{\bar{p}}$, 
		$\np{A_n}^{\bar{q}}\subseteq \np{A_n}^{\bar{p}}$, $g_n^{\bar{q}} \geq g_n^{\bar{p}}$, and $H_n^{\bar{q}}(\rho) \geq H_n^{\bar{p}}(\rho)$ for every $\rho \in A_n^{\bar{q}}$.
	\end{itemize}
	
	A condition $\bar{q}$ is a one-point extension of $\bar{p}$ if $\ell^{\bar{q}} = \ell^{\bar{p}} + 1$, $\bar{p}_n = \bar{q}_n$ for every $n \neq \ell^{\bar{p}}$, 
	and $\bar{q}_{\ell^{\bar{p}}} = \la \rho, g_{\ell^{\bar{p}}}^{\bar{p}}, H_{\ell^{\bar{p}}}^{\bar{p}}(\rho)\ra$ for some $\rho \in \np{A}_{\ell^{\bar{p}}}^{\bar{p}}$. We denote $\bar{q}$ by $\bar{p} \fr \la \rho\ra$. 
	Similarly, for a sequence of ordinals $\vec{\rho} = \la \rho_{\ell^{\bar{p}}}, \rho_{\ell^{\bar{p}}+1}, \dots , \rho_{m-1} \ra \in \prod_{\ell^{\bar{p}}\leq k < m} \np{A_k}^{\bar{p}}$, we define $\bp \fr \vec{\rho}$ to be the condition obtained by taking $m - \ell^{\bar{p}}$ consequtive one-point extensions by the ordinals in $\vec{\rho}$. 
	The ordering $\leq$ of $\bar{\po}$ is defined by setting $\bar{q} \geq \bar{p}$ if and only if $\bar{q}$ is obtained from $\bar{p}$ by finitely many one-point extensions and  direct extensions and one-point extensions. Equivalently, $\bar{q}$ is a direct extension of $\bar{p} \fr \vec{\rho}$ for some finite sequence $\vec{\rho} \prod_{\ell^{\bar{p}}\leq k < m} \np{A_k}^{\bar{p}}$ for some $m \geq \ell^{\bar{p}}$. 
	
	The following notational conventions and terminology will be useful for our treatment of $\bar{\po}$ and the revised extenders-based poset $\po$.
	\begin{enumerate}
		\item For every $\bp \in \bP$ and $n \geq \ell^{\bp}$, we define $\np{A}_{\bp\uhr n} = 
		\la \rho_0^{\bp},\dots,\rho^{\bp}_{\ell^{\bp}-1} \ra \times \prod_{\ell^{bp} \leq k < n}A^{\bp}_k$
		\item For every $\vec{\rho}_{n+1} = \la \rho_0,\dots,\rho_{n}\ra \in \np{A}_{\bp\uhr (n+1)}$, we define $\qo(\vec{\rho}_{n+1})$ to be the product of the Levy collapse posets which are determined by the sequence $\vec{\rho}_{n+1}$, namely, 
\[\col(\kappa_{-1}^{+2}, <\rho_0) \times \col(\rho_0^{+3},<\kappa_0) \times \dots  
 \times \col(\kappa_{n-1}^{+(n-1)+3}, <\rho_{n}) \times \col(\rho_{n}^{+n+3},<\kappa_{n}).\]
		
		Therefore, conditions in $\qo(\vec{\rho}_{n+1})$ are sequences of Levy collapse functions, of the form $\la g_0,h_0,\dots,g_n,h_n\ra$, where 
		for each $i$, $g_i \in \col(\kappa_{i-1}^{+i+2}, <\rho_{i})$ and $h_i \in \col(\rho_{i}^{+i+3}, <\kappa_i)$.
		\item We also define the restricted collapse product to be the poset $\qo'(\vec{\rho}_{n+1})$ which is obtained by removing the top collapse poset $\col(\rho_{n}^{+n+3},<\kappa_{n})$, from $\qo(\bro)$;
		\[\qo'(\bar{\rho}) =  \col(\kappa_{-1}^{+2}, <\rho_0) \times \col(\rho_0^{+3},<\kappa_0) \times \dots \times \col(\kappa_{n-1}^{+(n-1)+3}, <\rho_{n}) \]
		Clearly, $\qo(\bar{\rho}) = \qo'(\bar{\rho}) \times \col(\rho_{n}^{+n+3},<\kappa_{n})$. 
	\end{enumerate}
	
	Like the short extenders forcing $\po$, $\bP$ is a Prikry type forcing which admits some natural decomposition properties. We adopt the relevant notational conventions which were used to analyze $\po$. Therefore, for a condition ${\bp} = \la {\bp}_n \mid n < \omega\ra$ and $m < \omega$ we define ${\bp}\uhr m = \la {\bp}_n \mid n < m\ra$ and ${\bp}\dhr m =\la {\bp}_n \mid n \geq n\ra$. 
	We also define
	 $\bP_{< m} = \{ \bp\uhr m \mid \bp \in \bP\}$ and $\bP_{\geq m} = \{ \bp\dhr m \mid \bp \in bP\}$.
	The forcing $\bP/\bp$ breaks into the product 
	$\bP_{<\ell^{\bp}}/\bp\uhr \ell^{\bp} \times \bP_{\geq \ell^{\bp}}/\bp\dhr \ell^{\bp}$. We note that $\bP_{<\ell^{\bp}}/\bp\uhr \ell^{\bp} \cong \qo(\bro^{\bp})$ and that the direct extension order of $\bP_{\geq \ell^{\bp}}/\bp\dhr \ell^{\bp}$ is $\kappa_{\ell^{\bp}}$-closed. 
	
	A crucial component in the proof of the Prikry Lemma for $\bP$, is the ability to collect and amalgamate information from the different collapse posets $\qo(\vec{\rho}_m)$ (or $\qo'(\vec{\rho}_m)$) without deciding on the initial segment $\vec{\rho}_m$ of the generic Prikry sequence (e.g. see the proof of the Prikry Lemma in \cite{Unger-SEBFcol},\cite{SinUng-CombAleph})
	Isolating this part of the argument gives rise to the following assertion. 
	\begin{lemma}\label{lem-meetCollapse}
		Let $\bp \in \bP$ be a condition in $\bP$ and $n \geq \ell^{\bp}$. Suppose that
		$\{ D(\vec{\rho}_n) \mid \vec{\rho}_n \in \np{A}_{\bp\uhr n}\}$ is a family of sets so that each $D(\vec{\rho}_n)$ is a dense open in $\qo(\vec{\rho}_n)$. Then there exists a direct extension $\bq \geq^* \bp$ such that for every $\vec{\rho}_* \in \prod_{\ell^{\bp}\leq k < n}A^{\bq}_k$, the  condition $(\bq\uhr n) \fr \vec{\rho}_*$ belongs to $D(\vec{\rho}_{\bp}\fr \vec{\rho}_*)$. 
	\end{lemma}
	
	An important consequence of Lemma \ref{lem-meetCollapse} is that it is possible to reduce any $\bP$ name of an $\omega$-sequence of bounded sets in $\kappa_\omega$ to a family of names which depend on posets $\qo(\vec{\rho}_m)$ or $\qo'(\vro_m)$ for some suitable initial segments of $\vro_m$ of the generic Prikry sequence. 
	For example, suppose that $\la \name{S_n} \mid n < \omega\ra$ is a $\bP$ name for a sequence of sets so that $\name{S_n} \subseteq \rho_n^{+n+2}$.
	Let $\bp \in \po$, and note that for every $n \geq \ell^p$ and $\vec{\rho}_* \in \prod_{\ell \leq k \leq n}A^{\bp}_k$, the direct extension order of the 
	poset $\po_{\geq (n+1)}$, above the condition $(\bp \fr \vec{\rho}_*)\dhr (n+1)$, is $\kappa_n$ closed, and thus does not add new subsets to $\rho_n^{+n+2}$. 
	We can therefore assume that for every $n < \omega$, $\name{S_n}$ depends only on $\bP_{\leq n}$. 
	This name reduction can be further improved since the poset $\col(\rho_n^{+n+3},<\kappa_n)$ (which is the top collapse component of $\qo(\vec{\rho}_{n+1})$) is also sufficiently closed to decide all names of  subsets of $\rho_n^{+n+2}$.  
	It follows that for every $n \geq \ell^{\bp}$ and $\vro_{n+1} \in A_{\bp\uhr (n+1)}$, 
	there exists a dense open $D(\vro)$ of conditions in $\qo(\vec{\rho})$ which force $\name{S_n}$ to be equal to another name $\bar{S_n}$ which depends only on the restricted collapse product $\qo'(\vro)$.
	This allows us to apply Lemma \ref{lem-meetCollapse} and consequently, obtain the following result.
	
	\begin{corollary}\label{cor-PbarSn}
		Let $\la \name{S_n}\mid n < \omega\ra$ be a sequence of $\bar{\po}$-name so that each $\name{S_n}$ is a name for a subset of $\rho_n^{+n+2}$. 
		For every $\bar{p} \in \bar{\po}$ there exists $\bq \geq^* \bar{p}$ and a sequence of functions $\la \bar{S_n} \mid \ell^{\bar{p}} \leq n < \omega\ra$ so that 
		for every $m \geq \ell^{\bar{p}}$ and $\vro_* = \la \rho_{\ell^{\bar{p}}}, \dots, \rho_m \ra \in \prod_{\ell^{\bar{p}}< k \leq m} \np{A_k}^{\bq}$, 
		\[\bq \fr \vro_* \force \name{S_m} = \bar{S_m}(\vro_{\bar{p}}\fr \vro_*),\] where $\bar{S_m}(\vro)$ is a name of the restricted product $\qo'(\vro)$ for every $\vro \in \np{A}_{\bq\uhr (m+1)}$.
	\end{corollary}

	\subsection{A modified short extenders forcing}
	
	Before we turn to define our modified version $\po$ of the short extenders poset, let us point out a notational convention which we will use throughout this section frequently. 
	We will be working here with measure one sets with respect to measures $E_n(\alpha)$ for some $\alpha \in [\kappa_n,\kappa_n^{+n+2})$.
	It is clear that for every such $\alpha$, the measure $E_n(\alpha)$ contains the set $X_n$ of all ordinals $\nu < \kappa_n$ for which there exists a unique inaccessible cardinal $\rho$ such that $\nu \in [\rho, \rho^{+n+2})$. 
	Moreover, it is routine to verify that the map $\bar{\pi} : X_n \to \kappa_n$ which sends each $\nu \in X_n$ to $\rho$ as above, is a Rudin-Kiesler projection from $E_n(\alpha)$ to $E_n(\kappa_n)$.  We will $\bar{\pi} = \pi_{\alpha,\kappa_n}$ for every $\alpha < \kappa_n^{+n+2}$.
	The function $\bar{\pi}$ naturally extends to a map whose domain is a sequences of ordinals $\vec{\nu} = \la \nu_0,\dots, \nu_m\ra$ in $\prod_{k \leq m}X_k$ and as defined below, to a forcing projection from $\po$ to $\bP$. 
	We will frequently abuse the notation of $\bar{\pi}$, and further use it to denote its resulting natural extensions.
	
	\begin{definition}
	Conditions in the revised $\po$ are of the form $p = \la p_n \mid n < \omega\ra$ which satisfy the following conditions:
	\begin{enumerate}
		\item There exists some $\ell < \omega$ such that for every $n < \ell$, $p_n = \la f_n,\rho_n,g_n,h_n\ra$ where $\la \rho_n,g_n,h_n\ra$ satisfies condition (1) of the definition of $\bP$, and $f_n$ is a partial function from $\kappa_{\omega}^{++}$ to $\kappa_n$ of size $|f_n| \leq \kappa_\omega$. 
		\item For $n \geq \ell^p$, the components $p_n = \la a_n,\np{A_n}, A_n,f_n,g_n,H_n\ra$ are defined by induction on $n$.
		Suppose that $p\uhr n = \la p_k \mid k < n\ra$ has been defined and that $A_k \subseteq X_k$ for evey $k \geq \ell$. Let
		\[ \np{A}_{p\uhr n} = \{\rho_0\} \times \{\rho_1\} \times \dots \times \{\rho_{\ell-1}\} \times \prod_{\ell \leq k < n} \np{A_k}. \]
		$p_n = \la a_n,\np{A_n},A_n,f_n,g_n,H_n\ra$ is defined as follows.
		\begin{itemize}
			\item $\la \np{A_n}, g_n,H_n\ra$ satisfies condition (2) of the definition of $\bP$. 
			\item $f_n$ is a partial function from ${\kappa_\omega^{++}}$ to $\kappa_n$ of size $|f_n| \leq \kappa_\omega$.
			\item $a_n,A_n$ are functions. Their common domain is $\np{A}_{p\uhr n}$ and for every every $\vro \in \np{A}_{p\uhr n}$, the result of apllying $a_n$ and $A_n$ to $\bar{\rho}$ are denoted by $a_n^{\vro}$ and $A_n^{\vro}$ respectively. 
			Also, 
			we require that there exists some integer $k_n \geq 2$ so that for every $\vro \in \np{A}_{p\uhr n}$, $(a_n^{\bar{\rho}},A_n^{\bar{\rho}})$ is $k_n$-relevant pair in the sense of Definition \ref{def-relevant}.
			\item $\kappa_n \in \rng(a_n^{\vro})$ for every ${\vro} \in \np{A}_{p\uhr n}$ and $\pi_{\max(\rng(a_n^{{\vro}})), \kappa_n}``A_n^{\vro} = \np{A_n}$.
		\end{itemize}
		\item $\dom(a_n^{\vro_n}) \subseteq \dom(a_m^{\vro_m})$ whenever 
		$m \geq n \geq \ell$, $\vro_m \in \np{A}_{p\uhr m}$, $\vro_n \in \np{A}_{p\uhr n}$, and $\vro_n = \vro_m \uhr n$
		\item The sequence $\la k_n \mid n < \omega\ra$ is nondecreasing and unbounded in $\omega$.
	\end{enumerate}

	We denote $a_n,A_n,\np{A_n},f_n,g_n,h_n$ of $p$ by $a_n^p,A_n^p,\np{A_n}^p,f_n^p,g_n^p,h_n^p$ respectively. 
	Also, for $\vro_n \in \dom(a_n^p) = \dom(A_n^p)$, we denote for ease of notation
	$(a_n^p)^{\vro}$ and $(A_n^p)^{\vro}$ by $a_n^{p,\vro}$ and $A_n^{p,\vro}$ respectively.\\
	
	As before, the order relation $\leq$ of $\po$ is the derived from two basic operations of direct extension and one-point extension.
	\begin{enumerate}
		\item Given two conditions $p,q \in \po$, $q$ is a direct extension of $p$ if it satisfies the following conditions:
		\begin{itemize}
			\item $\ell^q = \ell^p$;
			\item $\rho_n^p = \rho_n^q$  and $h_n^p \leq h_n^q$ for every $n < \ell^p$;
			\item $f_n^p \subseteq f_n^q$ and $g_n^p \leq g_n^q$ for all $n < \omega$;
			\item for every $n \geq \ell^p$, $\np{A_n}^q \subseteq \np{A_n}^p$ and $H^p_n(\rho) \leq H^q_n(\rho)$ for all $\rho \in \np{A_n}^q$;
			\item for every $n \geq \ell^q$ and $\vro \in \np{A}_{q\uhr n}$, $(a_n^p)^{\vro} \subseteq (a_n^q)^{\vro}$. Furthermore, if
			$\gamma^{q,\vro}_n = \max(\rng(a_n^{q,\vro})$ and $\gamma^{p,\vro}_n = \max(\rng(a_n^{q,\vro}))$, then we require that
			$A_n^{q,\vro} \subseteq \pi_{\gamma^{q,\vro}_n,\gamma^{p,\vro}_n}^{-1}(A_n^{q,\vro})$. 
		\end{itemize}
		
		\item Given a condition $p = \la p_n \mid n < \omega\ra \in \po$, a one-point extension $p'$ of $p$ is a sequence $p' = \la p'_n \mid n < \omega\ra$ satisfying
		\begin{itemize}
			\item $\ell^{p'} = \ell^{p} + 1$;
			\item $p_n = p'_n$ for all $n \neq \ell^p$;
			\item denoting $\max(\dom(a_{\ell^p,\vro_p}))$ by $\eta$, there exists some $\nu\in A_{\ell^p}^{p,\vro_p}$ such that $p'_{\ell^p} = \la f_{\ell^p},\rho_{\ell^p},g_{\ell^p},h_{\ell^p}\ra$ where
			\begin{enumerate}
				\item $f^{p'}_{\ell^p} = f^{p}_{\ell^p} \cup \{ \la \tau, \pi_{a^{p,\vro_p}_{\ell^p}(\eta),a^{p,\vro_p}_{\ell^p}(\tau)}(\nu)\ra \mid \tau \in \dom(a^{p,\vro_p}_{\ell^p})\}$.
				\item $g^{p'}_{\ell^p} = g^p_{\ell^p}$.
				\item $\rho_{\ell^p} = \pi_{\eta,\kappa_{\ell^p}}(\nu) = \bar{\pi}(\nu)$.
				\item $h_{\ell^p}^{p'} =  H^{p}_{\ell^p}(\rho)$.
			\end{enumerate}
			\item For every $n > \ell^p$ we have that $f_n^{p'} = f_n$, $g_n^{p'}  = g_n$, $H_n^{p'} = H_n$, $a_n^{p'} = a_n^p\uhr\np{A}_{p'\uhr n}$, and $A_n^{p'} = a_n^p\uhr \np{A}_{p'\uhr n}$.
		\end{itemize}
		As usual, $p'$ is denoted by $p \fr \la \nu\ra$.
	\end{enumerate}
		\end{definition}

    We note that for every $p \in \po$ the 
	domain of $a^p_{\ell_p}$ and $A^p_{\ell^p}$ is the singleton $\np{A}_{p\uhr \ell_p} = \{ \vro_{p} \}$.

	\begin{definition}
		For a condition $p \in \po$, we define $\bar{\pi}(p)$ to be the sequence $\bar{p} = \la \bar{p}_n \mid n < \omega\ra$
		defined by $\bar{p}_n = \la \rho^p_n,g^p_n,h^p_n\ra$ for every $n < \ell^p$, and $\bar{p}_n = \la \np{A_n}^p,g^p_n,H^p_n\ra$ otherwise.
	\end{definition}

	\begin{remark}\label{remark-PprojPbar}
		The following facts are straightforward to derive from the definition of $\po$.
		\begin{itemize}
			\item  For every one-point extension $p \fr \la \nu\ra$ of $p$, $\bar{\pi}(p \fr \la \nu\ra)$ is the one-point extension of $\bar{\pi}(p) \fr \la \bar{\pi}(\nu)\ra$ of $\bar{p}$ in $\bar{\po}$. Conversely, for every ordinal $\rho$, if  $\bar{\pi}(p) \fr \la \rho\ra$ is a one-point extension of $\bar{\pi}(p)$ 
			there is $\nu \in A^{p,\vro_p}_{\ell^p}$ such that $\bar{\pi}(\nu)) = \rho$, and thus $\bar{\pi}(p \fr \la \nu\ra) = \bar{\pi}(p) \fr \la \rho\ra$.
			
			\item For every direct extension $p^*$ of $p$, $\bar{\pi}(p^*)$ is a direct extension of $\bar{\pi}(p)$ in $\bar{\po}$. Conversely, every direct extension of $\bar{\pi}(p)$ in $\bar{\po}$ is the $\bar{\pi}$ projection of some direct extension of $p$ in $\po$.
		\end{itemize}    
		It follows that $\pi : \po \to \bar{\po}$ is projection of Prikry type forcings from $\po$ onto $\bP$. 
	\end{remark}

	
	Finally, we modify the equivalence relation $\iff$ and the resulting ordering $\rar$.
	Given $p,q \in \po$ we write $p \iff q$ if and only if the following conditions hold:
	\begin{enumerate}
		\item $\ell^p = \ell^q$;
		\item $p\uhr \ell^p = q\uhr \ell^q$;
		\item $(g_n^p,H^n_p, \np{A}_{p\uhr n}) = (g_n^q,H^n_q,\np{A}_{q\uhr n})$ for all $n \geq \ell^p$; and
		\item there exists a nondecreasing sequence of integers $\la k_n^* \mid n < \omega\ra$ which is unbounded in $\omega$, such that for every $n \geq \ell^p$ and $\vro \in np{A}_{p\uhr n}$, $(a_n^{p,\vro},A_n^{p,\vro},f^p_n) \iff_{n,k^*n} (a_n^{q,\vro},A_n^{q,\vro},f^q_n)$ in the sense of Definition \ref{def-equiv}. 
	\end{enumerate}
	
	As before, we define $\rar$ on $\po$ as the closure of the operations $\leq$ and $\iff$.
	The modified short extenders forcing $\po$ satisfies all key properties which are satisfied by the poset $\po$ defined in the previous section.
	\begin{lemma}${}$
		\begin{enumerate}
			\item $(\po,\leq,\leq^*)$ is  Prikry type forcing.
			\item The identity function is a forcing projection from $(\po,\leq)$ to $(\po,\rar)$.
			\item $(\po,\rar)$ satisfies ${\kappa_\omega^{++}}$.c.c. 
		\end{enumerate}
	\end{lemma}
	The incorporation of interleaved collapse posets between the points of the $\kappa_n$s is standard. The fact that the the collapse poset between $\kappa_n$ and $\kappa_{n+1}$ is $\lambda_n^+$-closed (i.e., it is closed beyond the length of the extender $E_n$) is crucial for the proof of the Prikry Lemma. We refer the reader to \cite{Unger-EBFcollapse} for a proof of the Prikry Lemma for short extenders forcing with collapses. 
	The fact $(\po,\rar)$ satisfies ${\kappa_\omega^{++}}$.c.c follows from a similar argument to the one sketched in Section \ref{section-thm1}, where instead of using the type to replace $a_n= a_n^{p_\alpha}$ with an equivalent $a_n'$  once,
	we do it for $a_n^{\bar{\rho}}$ for every $\bar{\rho} \in \np{A}_{p\uhr n}$. 
	
	\subsection{Proof of Theorem \ref{thm2}}
	Suppos that $\la \kappa_n \mid n < \omega\ra$ is an increasing sequence of cardinals in $V$ and $n < \omega$, $j_n  :V_{\kappa_n+1} \to N_n$ is a $\kappa_n^{+n+3}$-strong embedding. 
	Let $E_n$ be the $(\kappa_n,\kappa_{n}^{+n+2})$ embedding derived from $j_n$. 
	Force with the modified short extenders forcing $(\po,\rar)$, and let $G \subseteq \po$ be a generic filer over $V$. 
	By Remark \ref{remark-PprojPbar}, the map $\bar{\pi} : \po \to \bar{\po}$ is a focring projection. Therefore the set $\bar{G} = \bar{\pi}``G \in V[G]$ is $\bP$ generic filter over $V$. 
	The intermediate generic extension $V[\bar{G}]$ contains the diagonal Prikry generic sequence $\vec{\rho} = \la \rho_n \mid n < \omega\ra$ and collapse generic filters for the intervals $(\kappa_{n-1}^{+n+3}, \rho_n)$ and $(\rho_n^{+n+3},\kappa_n)$, for every $n < \omega$. It is therefore clear that in $V[\bar{G}]$, the sequnece $\la \rho_n^{+n+2} \mid n < \omega\ra$ forms an infinite subset $\{ \omega_{s_n} \mid n < \omega\}$ of the set $\{ \omega_n \mid n < \omega\}$. 
	Fix a regular cardinal $\mu < \kappa_{\omega} = \aleph_\omega^{V[\bar{G}]}$ and $k < \omega$ so that $\rho_k > \mu$. We claim that for every sequence 
	$\vec{S} = \la S_n \mid k \leq n < \omega\ra$ of statonary sets $S_n \subseteq \rho_n^{+n+2} \cap \cof(\mu)$ is $V[\bar{G}]$, $\vec{S}$ is tightly stationary in $V[G]$. 
	We follow the steps of the proof of Theorem \ref{thm1} and fix, in $V$, 
	a sequence $\vec{a}^\mu$ of length ${\kappa_\omega^{++}}$, such that
	$S_{\vec{a}^\mu} \cap \cof(\mu)$ is stationary in ${\kappa_\omega^{++}}$. Since $(\po,\rar)$ satisfies ${\kappa_\omega^{++}}$.c.c,   $S_{\vec{a}^\mu} \cap \cof(\mu)$ remains stationary in $V[G]$. As explained in the proof of Theorem \ref{thm1}, the last property allows us to reduce the tight stationarity assertion concerning $\vec{S}$ in $V[G]$, to proving the following result: For every closed unbounded set $C \subseteq {\kappa_\omega^{++}}$, in $V$, and every condition $p \in \po$, there exists some $q^* \geq p$ and $\delta \in C \cap S_{\vec{a}^\mu} \cap \cof(\mu)$ such that 
	$q^* \force \delta \text{ is a continuity point of } \name{\vec{t}}, \text{ and } \name{t_\delta}(n) \in \name{S_n} \text{ for almost all } n < \omega.$
	
	To this end, we work in $V$ and fix a condition $p \in \po$ and a closed unbounded set $C \subseteq {\kappa_\omega^{++}}$. By extending $p$ if necessary, we may assume that  $\ell^p \geq k+1$, $\rho_k^p > \mu$, and  that
	$p \force  \forall n \geq k. \name{S_n} \subseteq \cof(\mu)$.
	By Corollary \ref{cor-PbarSn}, and the fact $\pi$ projects the poset $(\po,\leq^*)$ onto $(\bar{po},\leq^*)$ (Remark \ref{remark-PprojPbar}), 
	there exists a direct extension $q$ of $p$ and sequence of functions $\la \bar{S}_n \mid \ell^{p} \leq n < \omega\ra$ so that 
	for every $m \geq \ell^{p}$ and $\vro_* = \la \rho_{\ell^{p}}, \dots, \rho_m \ra \in \prod_{\ell^{\bar{p}}< n \leq m} \np{A_n}^{q}$, 
	$\bar{q} \fr \vro_* \force_{\bar{\po}} \name{S_m} = \bar{S}_m(\vro_p \fr \vro_*)$, where
	$\bar{S}_m(\bar{\rho})$ is a name of the restricted finite product $\qo'(\bar{\rho})$ of a stationary subset of $\rho_m^{+m+2}$.
	We point out that for each $m \geq \ell^p$, the domain of the function $\bar{S}_m$ coincides with the product $\np{A}_{q\uhr (m+1)} = \np{A}_{q\uhr m} \times \np{A}^{q}_m$. 
	
	We proceed to define sets of ordinals which will be needed for the construction of a desired extension $q^*$ of $q$.
	We first define a subset $a$ of ${\kappa_\omega^{++}}$ by
	\[a = \left(\cup_{n<\omega}\dom(f_n^q)\right) \bigcup \left(\cup_{n<\omega} \cup_{\bar{\rho}\in \np{q\uhr n}} \dom(a^{q,\bar{\rho}}_n)\right).\]
	$a$ has size $\kappa_\omega$ and is therefore bounded in ${\kappa_\omega^{++}}$.
	Similarly, for every $n \geq \ell^p$ we define a subset $r_n$ of $\kappa_n^{+n+2}$ by
	\[r_n = \bigcap_{\vro \in \np{A}_{q\uhr n}} \rng(a_n^{q,\bar{\rho}}).\] 
	$|r_n| < \kappa_n$ because $|\np{A}_{q\uhr n}| = \kappa_{n-1}$. In particular, $r_n$ is bounded in $\kappa_n^{+n+2}$.
	Next, we fix a increasing and continuous sequence $d = \la \delta(i) \mid i \leq \mu\ra$ of ordinals in ${\kappa_\omega^{++}} \setminus (\sup(a)+1)$, so that $\delta(\mu) \in S_{\vec{a}^\mu} \cap C$.  We also fix some $\tau \in {\kappa_\omega^{++}}$ above $\delta(\mu)$. In the constructing the final condition $q^*$, given below, we will add the collection $d \cup \{\tau\}$ to $\dom(a_n^{q^*,\vro})$ for every relevant $\vro$.
	
	$q^*$ is constructed in $\omega$ many steps. We will define for each $k \geq \ell^p$ a condition segments $q^k \in \po_{< k}$ and guarantee that the following conditions hold: (1) $q^k \geq^* q\uhr k$ for all $k \geq \ell^p$; and (2) $q^{k_2}\uhr k_1 \geq^* q^{k_1}\uhr k_1$ for every $k_1 < k_2$. 
	We start by taking $q^{\ell^p} = q\uhr \ell^p$. 
	Suppose that $q^n$ has been defined. Note that for every $\vro_n \in \np{A}_{q^n}$ and $\rho \in \np{A^q_n}$ the name $\bar{S}_n(\vro_n\fr \la \rho\ra)$ is defined. 
	Aplly $j_n$, and consider the function $T_n$ defined by $\dom(T_n) = \np{A}_{q^n}$ and  $T_n(\vro_n) = j_n(\bar{S}_n)(\vro_n \fr \la \kappa_n\ra)$. 
	By the elementarity of $j_n$, $T_n(\vro)$ is a $j_n(\qo')(\vro \fr \la \kappa_n\ra)$-name for a stationary subset of $\kappa_n^{+n+2}$. 
	Furthermore, the fact that $j_n : V_{\kappa_n+1} \to N_n$ is $\kappa_n^{+n+3}$-strong implies that $N_n$ contains every closed unbounded subset of $\kappa_n^{+n+2}$, and in particular, the set of all $n$-good ordinals below $\kappa_n^{+n+2}$. 
	It follows that for every $\vro_n \in \np{A}_{q^n}$, there is a dense open set $D(\vro_n) \subseteq \qo(\vro_n)$ of conditions  $z \in \qo(\vro_n)$ for which there it $g_z \in \col(\kappa_{n-1}^{n+2},<\kappa_n)$,  
	and a increasing and continuous sequence $d_n^{\vro_n} = \la \delta^{\vro_n}_n(i) \mid i \leq \mu\ra \subseteq \kappa_n^{+n+2} \setminus (\sup(r_n)+1)$ of $n$-good ordinals,
	such that 
	\[z \fr g_z \force_{\qo'(\vro_n\fr \la \kappa_n\ra)} \delta_n^{\vro_n}(\mu) \in T_n(\vro_n)\]
	Moreover, note that there are only $\kappa_{n-1}$ many relevant sequence $\vro_n$ and conditions $z \in \qo'(\vro_n)$, and the collapse poset $\col(\kappa_{n-1}^{n+2},<\kappa_n)$ is a $\kappa_{n-1}^{+n-2}$ closed. It is therefore routine to form a single condition $g_n^* \in  \col(\kappa_{n-1}^{n+2},<\kappa_n)$, extending $g^{q}_n$ (i.e., $g^q_n$ belongs to the $n$-th component of the condition $q$) such that  for all $\vro_n \in \np{A}_{q^n}$ and $z \in D(\vro_n)$, 
	\[z \fr g^*_n \force_{\qo'(\vro_n\fr \la \kappa_n\ra)} \delta_n^{\vro_n}(\mu) \in T_n(\vro_n).\]
	Let $\bq^n = \bar{\pi}(q^n) \in \bar{\po}_{<n}$. 
	By Lemma \ref{lem-meetCollapse} there exists a direct extension $\bt \geq^* \bq^n$
	such that for every $\vro_n \in \np{A}_{\bt}$, $\bt  \fr \vro_n \in D(\vro_n)$. 
	Let $t$ be a direct extension of $q^n$ in $\po_{<n}$ so that $\bar{\pi}(t) = \bt$, 
	and define $q^{n+1}\uhr n = t$. 
	It remains to define $q^{n+1}_n$. Before that, we define two auxiliery components $a_n^*$ and $A_n^*$ as follows:
	\begin{itemize}
		\item $\dom(a_n^*) = \np{A}_{t}$, and for every $\vro_n \in \dom(a^*_n)$, 
		let $a_n^{*,\vro_n} = a_n^{q,\vro} \cup \{ \la \delta(i),\delta_n^{\vro_n}(i) \ra \mid i \leq \mu \} \cup \{ \la \tau,\tau_n^{\vro}\ra\}$.
		Where, here, $\tau_n^{\vro_n}$ is some $n$-good ordinal which is an upper bound to the set $\rng(a_n^{q,\vro_n}) \cup \{ \delta^{\vro_n}_n(i) \mid i \leq \mu\}$ in the Rudin-Kiesler ordering $\leq_{E_n}$.
		
		\item $\dom(A^*_n) = \np{A}_{t}$. For every $\vro_n \in \dom(A^*_n)$,
		$A_n^{*,\vro_n}$ is defined to be the set of all $\nu \in \pi_{\tau_n^{\vro_n},\max(a_n^{q,\vro_n})}^{-1}(A_n^{q,\vro_n})$ which satisfy the following conditions:
		\begin{enumerate}
			\item The condition $g^*_n \in \col(\kappa_{n-1}^{+n+2},<\kappa_n)$ belongs to $\col(\kappa_{n-1}^{+n+2},<\rho_\nu)$, where $\rho_\nu = \bar{\pi}(\nu)$;
			\item For every $\bro \in \np{A}_{q'}$, the condition $(q' \fr \bro) * g^*_n$ of $\qo(\bro) \times \col(\kappa_{n-1}^{+n+2},<\rho_\nu) = \qo'(\bro\fr \la \rho_\nu\ra)$ forces the statements
			\[``\pi_{\tau_n^{\bro},\delta_n^{\bro}(\mu)}(\nu) \in \bar{S_n}(\bro \fr \la \rho_\nu\ra)`` \]
			and
			\[`` \la \pi_{\tau_n^{\bro},\delta_n^{\bro}(i)}(\nu) \mid i \leq \mu\ra \text{ is an increasing and continuous sequence } `` \]
		\end{enumerate}
	\end{itemize}
	It is straightforward to verify that $a_n^*$ and $A_n^*$ are well defined components, using our choice of $\bt$ and the sequences $\la \delta_n^{\vro_n}(i) \mid i \leq \mu\ra$, $\vro_n \in \np{A}_{\bt}$. 

	We turn to define $q^{n+1}_n = \la a_n,\np{A_n},A_n,f_n,g_n,H_n \ra$. We define each component in turn:
	\begin{enumerate}
		\item $a_n = a_n^*$;
		\item $\np{A_n} = np{A_n}^q \cap \left(\bigcap_{\vro_n \in \np{A}_{t}}\bar{\pi}``A_n^{*,\vro_n}\right)$;
		\item $\dom(A_n) = \np{A}_{t}$ and for every $\vro_n \in \dom(A_n)$, $A_n^{\vro} = A_n^{*,\vro_n} \cap \bar{\pi}^{-1}(\np{A_n})$;
		\item $f_n = f_n^q$;
		\item $g_n = g_n^*$;
		\item $H_n = H_n^q$.
	\end{enumerate}
	It is straightforward to verify that $q^*$ is a direct extension of $q$ in $\po$. Our choice of sequences $\la \delta_n^{\vro_n}(i) \mid i \leq \mu\ra$, $\vro_n \in \np{A}_{\bt}$ for every $n \geq \ell^{p}$ and $\vro_n \in \np{A}_{q^*\uhr n}$ guarantee that for every $n \geq \ell^{q}$ and a sequence
	$\vec{\nu} = \la \nu_{\ell^q}, \dots \nu_n\ra$, if $q^* \fr \vec{\nu}$ is a valid extension on of $q^*$ then it must force that $t_{\delta(\mu)}(n) \in \name{S_n}$ is a limit point of the increasing sequence $\la t_{\delta(i)}(n) \mid i < \mu\ra$.
	We conclude that $q^*$ forces that $t_\delta$ is a continuity point of $\vec{t}$ and that $t_\delta(n) \in \name{S_n}$ for almost all $n < \omega$. 
	 \qed{Theorem \ref{thm2}}
	
	\providecommand{\bysame}{\leavevmode\hbox to3em{\hrulefill}\thinspace}
	\providecommand{\MR}{\relax\ifhmode\unskip\space\fi MR }
	\providecommand{\MRhref}[2]{%
		\href{http://www.ams.org/mathscinet-getitem?mr=#1}{#2}
	}
	\providecommand{\href}[2]{#2}

\end{document}

\cite{Gitik-EBF1}
\cite{Gitik-HB}
\cite{GitUng-SEF}
\cite{Unger-SEBFcol}
\cite{CFM-CanI}
\cite{CFM-CanStrII}
\cite{She-CarAri}
\cite{ForMag-MS}
\cite{ForMag-veryweak}
\cite{Baum}
\cite{chen-treelikeTS}
\cite{CheNee-TS}
\cite{BN-MSI}
\cite{Unger-EBFcollapse}
